\documentclass{amsart}
\usepackage{amsmath}
\usepackage{amsthm}
\usepackage{amssymb}
\usepackage[margin=0.5in]{caption}
\usepackage{dsfont}
\usepackage{enumitem}
\usepackage{graphicx}
\usepackage[hidelinks]{hyperref}
\usepackage[capitalise]{cleveref}
\usepackage{mathrsfs}
\usepackage{mathtools}
\usepackage{nameref}
\usepackage[new]{old-arrows}
\usepackage{stmaryrd}
\usepackage{thmtools}
\usepackage{url}
\usepackage{xcolor}
\usepackage{xparse}

\usepackage{tikz}
\usetikzlibrary{cd}
\usetikzlibrary{positioning, angles, quotes}

\theoremstyle{plain}
\newtheorem{thm}{Theorem}[section]
\newtheorem*{thm*}{Theorem}

\newtheorem{cor}[thm]{Corollary}
\newtheorem*{cor*}{Corollary}

\newtheorem{prop}[thm]{Proposition}
\newtheorem*{prop*}{Proposition}

\newtheorem{lem}[thm]{Lemma}
\newtheorem*{lem*}{Lemma}

\newtheorem{claim}{Claim}[thm]
\newtheorem*{claim*}{Claim}

\newtheorem*{exer*}{Exercise}

\newtheorem*{conj*}{Conjecture}

\theoremstyle{definition}
\newtheorem{defn}[thm]{Definition}
\newtheorem*{defn*}{Definition}

\newtheorem*{ex*}{Example}

\newtheorem*{q*}{Question}

\theoremstyle{remark}
\newtheorem{rem}[thm]{Remark}
\newtheorem*{rem*}{Remark}

\theoremstyle{plain}

\Crefname{thm}{Theorem}{Theorems}
\Crefname{defn}{Definition}{Definitions}
\Crefname{lem}{Lemma}{Lemmata}

\newcommand{\gpd}{\mathcal{G}}
\newcommand{\Hgpd}{\mathcal{H}}
\newcommand{\essbdd}{L^{\infty}(\gpd^0)}

\newcommand{\VN}{\mathcal{N}}
\newcommand{\UO}{\mathcal{U}}

\newcommand{\C}{\mathbb{C}}
\newcommand{\N}{\mathbb{N}}

\newcommand{\R}{\mathbb{R}}

\DeclareMathOperator{\coker}{coker}

\DeclareMathOperator{\im}{im}

\DeclareMathOperator{\rk}{rk}

\DeclareMathOperator{\Tor}{Tor}

\DeclareMathOperator{\trace}{tr}
\DeclareMathOperator{\res}{res}
\DeclareMathOperator{\Mat}{Mat}
\DeclareMathOperator{\pr}{pr}

\newcommand{\vertii}[1]{{\left\vert\kern-0.25ex\left\vert #1 \right\vert\kern-0.25ex\right\vert}}
\newcommand{\vertiii}[1]{{\left\vert\kern-0.25ex\left\vert\kern-0.25ex\left\vert #1 \right\vert\kern-0.25ex\right\vert\kern-0.25ex\right\vert}}

\makeatletter
\newsavebox{\@brx}
\newcommand{\llangle}[1][]{\savebox{\@brx}{\(\m@th{#1\langle}\)}%
  \mathopen{\copy\@brx\mkern2mu\kern-0.9\wd\@brx\usebox{\@brx}}}
\newcommand{\rrangle}[1][]{\savebox{\@brx}{\(\m@th{#1\rangle}\)}%
  \mathclose{\copy\@brx\mkern2mu\kern-0.9\wd\@brx\usebox{\@brx}}}
\makeatother

\title[On Vanishing Criteria of \texorpdfstring{$L^2$}--Betti Numbers of Groups]{On Vanishing Criteria of \texorpdfstring{$L^2$}--Betti Numbers of Groups}

\author{Pablo S\'anchez-Peralta}
\address[P.~ S\'anchez-Peralta]{Universidad Aut\'onoma de Madrid, Madrid, Spain}
\email{pablo.sanchezperalta@uam.es}

\begin{document}

\maketitle

\begin{abstract}
    Let $G$ be a countable group and $k$ a positive integer, we show that the $L^2$-Betti numbers of the group $G$ vanish up to degree $k$ provided that there is some infinite index subgroup $H$ with finite $k$th $L^2$-Betti number containing a normal subgroup of $G$ whose $L^2$-Betti numbers are all zero below degree $k$. This generalizes previous criteria of both Sauer and Thom, and Peterson and Thom. In addition, we exhibit a purely algebraic proof of a well-known theorem of Gaboriau concerning the first $L^2$-Betti number which was requested by Bourdon, Martin and Valette. Finally, we provide evidence of a positive answer for a question posted by Hillman that wonders whether the above statement holds for $k=1$ and $H$ containing a subnormal subgroup instead.
\end{abstract}

\section{Introduction}

    Given $G$ a countable group, the $L^2$-Betti numbers $\beta_i^{(2)}(G)$ have proved to be a powerful asymptotic invariant. The purpose of this article is to extend some of the vanishing criteria for these invariants. Specifically, we shall see how the $L^2$-Betti numbers of certain subgroups take control over the $L^2$-Betti numbers of the whole group. One instance of such phenomenon is when the subgroup is of finite index. This situation is completely well understood (see \cite[Theorem 4.15]{Kamm_l2invt}) due to the multiplicity of these invariants. However, the interaction with infinite index subgroups is still a challenging issue so far. One of the first approaches to attack this question was given by D. Gaboriau in his groundbreaking article \cite{Gab02}. The strategy implemented there was based on the study of the notion of $L^2$-Betti numbers of measured equivalence relations, which were introduced by Gaboriau in the same paper, to show that whenever the infinite index subgroup is normal, infinite and with finite first $L^2$-Betti number, then $\beta_1^{(2)}(G)$ must be zero \cite[Théorème 6.8]{Gab02}.

    In light of Gaboriau's theorem, it is natural to ask whether the vanishing of higher $L^2$-Betti numbers of a group G can be detected by looking at higher $L^2$-Betti numbers of the normal subgroup $N$. This was done by R. Sauer and A. Thom in \cite{SauerThom}. Concretely, they constructed a spectral sequence for $L^2$-type cohomology groups which shows that Gaboriau's result is the degree $1$ case of the next general statement. Recall that the zeroth $L^2$-Betti number of a group is zero if and only if the group is infinite. The statement reads as follows: Let $G$ be a countable group, $N$ a normal subgroup in $G$ and $k$ a positive integer such that $\beta_i^{(2)}(N)$ vanishes for all $0\leq i\leq k-1$ and $\beta_k^{(2)}(N)$ is finite; then $\beta_k^{(2)}(G)$ vanishes. It is worth mentioning that L\"uck \cite[Theorem 7.2.(6)]{Luck02} proved this criterion whenever the quotient group has a non-torsion element or finite subgroups of arbitrarily high order, and likewise in \cite[Corollary 1]{BMV05} a proof of the former case was exhibited. In \cite{BMV05} and \cite{SauerThom}, the authors claimed that {\it it is a challenging and vaguely irritating question to find a purely cohomological proof of Gaboriau's result} and that {\it there is no proof that does not use measured equivalence relations}, respectively. In this paper we remedy this situation stating a completely algebraic proof pending on the next result. Observe that in order to recover the original statement one can take $M=\C$ with the trivial action (see \cref{subsec: von_Neumann_dim} for the precise definitions).

    \begin{thm}\label{thm: module_beta_k}
        Let $G$ be a countable group, $N$ an infinite index normal subgroup in $G$, $M$ a left $\C[G]$-module with $N$ acting trivially and $k$ a positive integer. Suppose that $\beta_i^{(2)}(N)$ vanishes for all $0\leq i\leq k-1$ and $\dim_{\scalebox{0.8}{$\UO(N)$}}(\Tor_k^{\scalebox{0.8}{$\C[N]$}}(\UO(N),M))$ is finite. Then 
        \[
        \dim_{\scalebox{0.8}{$\UO(G)$}}\left(\Tor_k^{\scalebox{0.8}{$\C[G]$}}(\UO(G),M)\right)=0.
        \]
    \end{thm}
    
    In \cite{PetersonThom_Freiheit}, J. Peterson and A. Thom extended Gaboriau's result in a different direction, namely relaxing the normal hypothesis. In particular, they showed that given $G$ a countable group and $H$ an infinite index subgroup in $G$ containing an infinite normal subgroup $N$ in $G$, then the fact that $\beta_1^{(2)}(H)$ is finite is sufficient to ensure that the first $L^2$-Betti number of $G$ vanishes (see \cite[Theorem 5.12]{PetersonThom_Freiheit} for the general statement). This time we cannot carry out the module approach of \cref{thm: module_beta_k} since it builds heavily on the existence of the intermediate ring $\UO(N)*G/N$ lying between $\UO(N)$ and $\UO(G)$. However, we can look carefully at the homological properties of the groupoid ring and translate the essence of the algebraic proof to give a higher degree version of the Peterson and Thom's normal case.

    \begin{thm}\label{thm: groupoids_beta_k}
        Let $G$ be a countable group, $H$ an infinite index subgroup in $G$ containing a normal subgroup $N$ in $G$ and $k$ a positive integer. Suppose that $\beta_i^{(2)}(N)$ vanishes for all $0\leq i\leq  k-1$ and $\beta_k^{(2)}(H)$ is finite. Then $\beta^{(2)}_k(G)$ vanishes.
    \end{thm}

    Furthermore, on the one side L\"uck proved that having a normal subgroup with vanishing $L^2$-Betti numbers up to degree $k$ for some non-negative integer $k$, implies that the $L^2$-Betti numbers of the group must be zero up to degree $k$ \cite[Theorem 7.2.(2)]{Luck02}. On the other side, Peterson and Thom showed that for $k=0$ the situation is better. Specifically, given $G$ a countable group and $H$ a subgroup in $G$ containing an infinite normal subgroup $N$ in $G$, then the first $L^2$-Betti number of $H$ bounds that of $G$ (see \cite[Theorem 5.6]{PetersonThom_Freiheit} for the complete statement). In this article, we show that there is nothing special about the degree zero case, and we extend this result to higher degrees and even relax the normal condition.

    \begin{thm}\label{thm: control_beta_k_subnormal_for_H}
         Let $G$ be a countable group, $H$ a subgroup in $G$ containing a subnormal subgroup $N$ in $G$ and $k$ a positive integer. Suppose that $\beta_i^{(2)}(N)$ vanishes for all $0\leq i\leq  k-1$. Then $\beta_k^{(2)}(H)\geq \beta_k^{(2)}(G)$.
    \end{thm}

    Finally, we mention that J. A. Hillman ran into an interesting question concerning the first $L^2$-Betti number of finitely generated groups in \cite{Hillman_Subnormality} which is used as hypothesis:

    {\it If a finitely generated group $G$ has infinite subgroups $N\leq U$ such that $N$ is subnormal in $G$, $U$ is finitely generated and $[G:U]$ is infinite, then $\beta_1^{(2)}(G)=0$.}

    The above statement (see [\cite{Hillman_Subnormality}, Hypothesis $B$]) enabled Hillman to develop a new homological approach towards some old results on 3-manifold groups due to Elkalla. We did not manage to settle this question in full generality, though \cref{thm: control_beta_k_subnormal_for_H} along with the next statement for $k=1$ provide evidence of a positive answer (in our results $H$ plays the role of $U$ in the above question). 

    \begin{cor}\label{cor: Hillman_depth_2}
        With the above notation, \cref{thm: groupoids_beta_k} still holds if $N$ is a subnormal subgroup in $G$ of depth at most $2$.
    \end{cor}

    Among the consequences, this covers classical results by A. Karras and D. Solitar in \cite{KarrasSolitar}, H. Griffiths in \cite{Griffiths}, and B. Baumslag in \cite{Baumslag_FreeProd}, as well as Theorem 3.1 and Corollary 3.4 in \cite{BridsonHowie} by M. Bridson and J. Howie.

\subsection*{Organization of the paper}

    In \cref{sec: prelims} we recall some notions that will appear throughout the paper. In \cref{sec: hom_alg_dim_theory} we study how the dimension of the functor $\Tor$ is affected by changing dimension-isomorphic modules or rings; the results shown in this section are geared towards proving dimension homological aspects of groupoids. In \cref{sec: hom_features_gpd} we show that the groupoid ring is also free as a module over the groupoid ring of certain finite index subgroupoids, and we state a dimension analogue of Shapiro's lemma. In \cref{sec: algebraic_proofs} we show the module theoretical approach proving Gaboriau's and Sauer and Thom's result with algebraic methods. In \cref{sec: new_vanishing_crit} we prove the new vanishing criteria on $L^2$-Betti numbers of groups, which build on the results of two previous sections, as well as the control theorems aimed to answer Hillman's question.

\subsection*{Acknowledgments} 

    The author is grateful to Andrei Jaikin-Zapirain for many helpful conversations and for having suggested him \cref{thm: finite_dim_ind_zero_dim} which has been the starting point of this work. The author is supported by PID2020-114032GB-I00 of the Ministry of Science and Innovation of Spain.

\section{Preliminaries}\label{sec: prelims}

    Throughout, groups are countable, rings are assumed to be associative and unital, ring homomorphisms preserve the unit and modules are left modules unless otherwise specified.

    \subsection{Ore localization}

        Let $R$ be a ring and $T\subseteq R$ a multiplicative set. We say that $R$ {\it satisfies the left Ore condition with respect to} $T$ if for every $t\in T$ and $r\in R$, $Tr\cap Rt\neq \emptyset$. Whenever this condition holds we can embed $R$ into a ring $S$, namely {\it a left ring of fractions} that we denote by $T^{-1}R$, in which every element of $T$ is invertible and every element of $S$ can be written as $t^{-1}r$ for some $t\in T$ and $r\in R$ (\cite[Theorem 6.2]{GW04}).

        We record here a useful feature of these left rings of fractions.

        \begin{lem}{\cite[Lemma 6.1]{GW04}}\label{lem: common_denominator_ore}
            Suppose $R$ satisfies the left Ore condition with respect to $T$. Then given any $x_1,\ldots, x_n\in T^{-1}R$, there exist $r_1,\ldots, r_n\in R$ and $t\in T$ such that each $x_i$ equals to $t^{-1}r_i$ for $i=1,\ldots, n$.
        \end{lem}

    \subsection{Homological algebra}
    
        Let $R\subseteq S\subseteq T$ be rings and $0\rightarrow M_1 \rightarrow M_2\rightarrow M_3\rightarrow 0$ a short exact sequence of left $R$-modules. For each non-negative integer $k$, taking the associated long exact sequences, we get a commutative diagram
        \[
        \begin{tikzcd}
            \Tor_k^R(S,M_2)\arrow{d}\arrow{r}{i_k}&\Tor_k^R(S,M_3)\arrow{d}\\
            \Tor_k^R(T,M_2)\arrow{r}{j_k}&\Tor_k^R(T,M_3)
        \end{tikzcd}.
        \]
        Not only that, since $S\subseteq T$ we have a commutative diagram 
        \begin{equation}\label{lem: commutative_diagram_Tor}
            \begin{tikzcd}
                T\otimes_S \Tor_k^R(S,M_2)\arrow{d}\arrow{r}&T\otimes_S \im i_k\arrow{d}\\
                \Tor_k^R(T,M_2)\arrow{r}&\im j_k
            \end{tikzcd}
        \end{equation}
        with surjective horizontal maps. This commutative diagram will play a central role in the proof of \cref{thm: module_beta_k}. 
        
        For the reader's convenience, we record a result which we will referred to as Shapiro's lemma from now on.

        \begin{lem}{\cite[Corollary 10.61]{Rotman09}}
            Let $S$ be a ring and $R$ a subring of $S$ such that $S$ is flat as right $R$-module. Then, for any right $S$-module $A$, for any left $R$-module $B$ and for any non-negative integer $k$, it holds
            \[
            \Tor_k^R(A,B)\cong \Tor_k^S(A,S\otimes_R B).
            \]
        \end{lem}

        Recall that a ring $R$ is called {\it left} (resp. {\it right}) {\it semihereditary} if every finitely generated submodule of a projective left (resp. right) $R$-module is projective. A ring $R$ is called {\it semihereditary} if $R$ is both left and right semihereditary. Semihereditary rings enjoy the following property.

        \begin{prop}{\cite[Theorem 4.32]{Rotman09}}\label{prop: submod_flat_is_flat_if_semihered}
            Let $R$ be a semihereditary ring, then every submodule of a flat $R$-module is flat.
        \end{prop}

        In \cite[Theorem 6.7 on pp. 239 and 288]{Luck02} it is shown that:

        \begin{thm}\label{thm: von_Neumann_alg_is_semihered}
            Every von Neumann algebra $\VN$ is a semihereditary ring.
        \end{thm}

        For finite von Neumann algebras there is another situation in which flat modules appear. 

        \begin{prop}{\cite[Theorem 1.48]{SauerGroupoids}}\label{prop: trace_preserving_flat}
            Any trace-preserving $*$-homomorphism between finite von Neumann algebras is a faithfully flat ring extension.
        \end{prop}

    \subsection{Left ideals in group algebras}

        Let $G$ be a group, we denote by $I_G$ the augmentation ideal of the group algebra $\C[G]$. If $H$ is a subgroup of $G$ we denote by $I_H^G$ the left ideal of $\C[G]$ generated by $I_H$. We recall the following standard result (see, for instance, \cite[Lemma 2.1]{Jaikin_freeqgp}).

        \begin{lem}\label{lem: isomph_ind_augment_ideal}
            Let $N\leq H$ be subgroups of $G$. Then the following holds.
            \begin{enumerate}[label=(\roman*)]
                \item The canonical map $\C[G]\otimes_{\scalebox{0.8}{$\C[H]$}} I_H\rightarrow I_H^G$ sending $a\otimes b$ to $ab$, is an isomorphism of left $\C[G]$-modules.
                \item The canonical map $\C[G]\otimes_{\scalebox{0.8}{$\C[H]$}}I_H/I_N^H\rightarrow I_H^G/I_N^G$, sending $a\otimes (b+I_N^H)$ to $ab+I_N^G$, is an isomorphism of left $\C[G]$-modules.
            \end{enumerate}
        \end{lem}

    \subsection{Sylvester rank functions}
    
        \begin{defn}\label{defn: SMRF}
            A {\it Sylvester module rank function} or simply {\it SMRF} $\dim$ on $R$ is a map that assigns a non-negative real number to each finitely presented left $R$-module and that satisfies the following properties:
            \begin{enumerate}[label=(\roman*)]
                \item $\dim(0)=0$ and $\dim(R)=1$;
                \item $\dim(M_1\oplus M_2)=\dim(M_1)+\dim(M_2)$; and
                \item if $M_1\rightarrow M_2\rightarrow M_3\rightarrow 0$ is exact then
                \[
                \dim(M_1)+\dim(M_3)\geq \dim(M_2)\geq \dim(M_3). 
                \]
            \end{enumerate}
        \end{defn}

        In \cite{LiSMRF}, H. Li extends each Sylvester module rank function to all pairs of left $R$-modules $M_1\subseteq M_2$.
        
        \begin{defn}\label{defn: BSMRF}
            A {\it bivariant Sylvester module rank function} or simply {\it BSMRF} $\dim(\cdot$\textbar$\cdot)$ on $R$ is an $\R_{\geq 0}\cup \{+\infty\}$-valued function $(M_1,M_2)\mapsto \dim(M_1$\textbar$M_2)$ on the class of all pairs of $R$-modules $M_1\subseteq M_2$ satisfying the following conditions:
            \begin{enumerate}[label=(\roman*)]
                \item $\dim(M_1$\textbar $M_2)$ is an isomorphism invariant;
                \item (normalization) setting $\dim(M)=\dim(M$\textbar $M)$ for all $R$-modules $M$, one has $\dim(0)=0$ and $\dim(R)=1$;
                \item (direct sum) for any $R$-modules $M_3\subseteq M_4$, one has
                \[
                \dim(M_1\oplus M_3\mbox{\textbar} M_2\oplus M_4)=\dim(M_1\mbox{\textbar} M_2)+\dim(M_3\mbox{\textbar} M_4);
                \]
                \item (continuity) $\dim(M_1$\textbar $M_2)=\sup_{M_1'}\dim(M_1'$\textbar $M_2)$ for $M_1'$ ranging over all finitely generated $R$-submodules of $M_1$;
                \item (continuity) if $M_1$ is finitely generated, $\dim(M_1$\textbar $M_2)=\inf_{M_2'}\dim(M_1$\textbar $M_2')$ for $M_2'$ ranging over all finitely generated $R$-submodules of $M_2$ containing $M_1$; and
                \item (additivity) $\dim(M_2)=\dim(M_1$\textbar $M_2)+\dim(M_2/M_1)$.
            \end{enumerate}
        \end{defn}

        As aforementioned, this notion is the right way to extend a SMRF $\dim$ so that we can compute dimensions of arbitrary modules. Briefly, first one can set the dimension of a finitely generated module $M$ as $\inf_{M'}\dim(M')$ for $M'$ ranging over all finitely presented $R$-modules which admit $M$ as a quotient module, which extends $\dim$ for finitely presented $R$-module (see \cite[Lemma 3.8]{LiSMRF}). Second, for any pair of finitely generated $R$-modules $M_1\subseteq M_2$ define $\dim(M_1$\textbar $M_2)=\dim(M_2)-\dim(M_2/M_1)$. Finally, extend it to arbitrary modules via \cref{defn: BSMRF}.

        \begin{thm}{\cite[Theorem 3.3]{LiSMRF}}\label{thm: SMRF_to_BSMRF}
            Every Sylvester module rank function for $R$ extends uniquely to a bivariant Sylvester module rank function for $R$.
        \end{thm}

        Moreover, this BSMRF has desirable continuity properties.
        
        \begin{lem}{\cite[Proposition 3.23]{LiSMRF}}\label{lem: continuity_dim}
            Let $\dim(\cdot$\textbar$\cdot)$ be a bivariant Sylvester module rank function for $R$. The following holds.
            \begin{enumerate}[label=(\roman*)]
                \item $\dim(M_1$\textbar $M_2)$ is increasing in $M_1$, that is, for any $R$-module $M_1\subseteq M_1'\subseteq M_2$, one has $\dim(M_1$\textbar $M_2)\leq \dim(M_1'$\textbar$M_2)$.
                \item $\dim(M_1$\textbar $M_2)$ is decreasing in $M_2$, that is, for any $R$-module $M_1\subseteq M_2\subseteq M_2'$, one has $\dim(M_1$\textbar $M_2)\geq \dim(M_1$\textbar$M_2')$.
            \end{enumerate}
        \end{lem}

        It is a well-known fact that if $\varphi:R\rightarrow T$ is a ring homomorphism and $T$ has a SMRF $\dim_T$ then $R$ inherits a SMRF, namely $\dim_R (M):=\dim_T (T\otimes_R M)$ for every finitely presented left $R$-module $M$. This is the so-called {\it induced} SMRF from $T$ to $R$. Similarly, whenever $T$ has a BSMRF so does $R$. In this direction, for finite von Neumann algebras one has the following:

        \begin{prop}{\cite[Theorem 2.6]{SauerGroupoids}}\label{prop: dim_extension_for_finite_vNa}
            Let $\phi:\mathcal{A}\rightarrow\mathcal{B}$ be a trace-preserving $*$-homomorphism between finite von Neumann algebras $\mathcal{A}$ and $\mathcal{B}$. Then for every $\mathcal{A}$-module $M$ we have
            \[
            \dim_{\mathcal{A}}(M)=\dim_{\mathcal{B}}(\mathcal{B}\otimes_{\mathcal{A}}M).
            \]
        \end{prop}
        
        In the forthcoming sections most of the SMRF we will be working with are {\it exact}, that is, if $0\rightarrow M_1\rightarrow M_2\rightarrow M_3\rightarrow 0$ is an exact sequence of left modules then $\dim(M_1)+\dim(M_3)=\dim(M_2)$. As a matter of fact (see \cite[Proposition 4.2]{LiSMRF}), whenever a BSMRF is exact on finitely presented modules, then for any $R$-modules $M_1\subseteq M_2$ one has $\dim(M_1$\textbar $M_2)=\dim(M_1)$. Furthermore, for any $R$-module $M$ $\dim(M)$ can be computed as $\sup_{M'}\dim(M')$ for $M'$ ranging over the finitely generated $R$-submodules of $M$. In particular, this applies to von Neumann regular rings with SMRF, since all finitely presented modules are projective (cf. \cite[Theorem 4.9 \& Theorem 3.63]{Rotman09}), and to finite von Neumann algebras \cite[Theorem 6.7]{Luck02}.

        On the other hand, there is a dual notion to that of Sylvester module rank function. Let $R$ be a ring. A \textit{Sylvester matrix rank function} $\rk$ on $R$ is a function that assigns a non-negative real number to each matrix over $R$ and satisfies the following conditions:
        \begin{enumerate}[label=(\roman*)]
            \item $\rk(A)=0$ if $A$ is any zero matrix and $\rk(1)=1$;
            \item $\rk(AB)\leq \min\{\rk(A),\rk(B)\}$ for any matrices $A$ and $B$ which can be multiplied;
            \item $\rk\begin{psmallmatrix}
                 A & 0 \\
                 0 & B
            \end{psmallmatrix}=\rk(A)+\rk(B)$ for any matrices $A$ and $B$; and
            \item $\rk \begin{psmallmatrix}
                 A & C \\
                 0 & B
            \end{psmallmatrix} \geq \rk(A)+\rk(B)$ for any matrices $A,B$ and $C$ of appropriate sizes.
        \end{enumerate}

        As one might expect, there is a bijective correspondence between this two notions. Specifically, given a Sylvester matrix rank function $\rk$ we can define a Sylvester module rank function $\dim$ by assigning to any finitely presented left $R$-module with presentation $M=R^m/R^nA$ for some $A\in \Mat_{n\times m}(R)$, the value $\dim(M):=m-\rk(A)$; and conversely, we get a Sylvester matrix rank function by setting $\rk(A):=m-\dim(R^m/R^nA)$ (cf. \cite[Theorem 4]{Malcolmson} and \cite[Proposition 1.2.8]{DLopezAlvarezThesis}).

    \subsection{Von Neumann dimension}\label{subsec: von_Neumann_dim}

        A countable group $G$ acts by left and right multiplication on $\ell^2(G)$. We can extend the right action of $G$ on $\ell^2(G)$ to the group algebra $\C[G]$; and hence consider $\C[G]$ as a subalgebra of the bounded linear operators on $\ell^2(G)$, namely $\mathcal{B}(\ell^2(G))$.

        A finitely generated {\it Hilbert} $G$-module is a closed subspace $V\leq \ell^2(G)^n$ invariant under the left action of $G$. Given such $V$, there exists $\pr_V$, the orthogonal projection onto $V$. We set 
        \[
        \dim_G(V):=\trace_G(\pr_V):=\sum_{i=1}^n \langle \pr_V(1_i),1_i\rangle_,
        \]
        where $1_i$ is the element of $\ell^2(G)^n$ having $1$ in the $i$th entry and $0$ in the rest of the entries. We call this number the {\it von Neumann dimension} of $V$. This dimension induces a Sylvester matrix rank function on $\C[G]$. Concretely, given $A\in \Mat_{n\times m}(\C[G])$ we can then associate the natural bounded linear operator $\phi_G^A:\ell^2(G)^n\rightarrow\ell^2(G)^m$ given by right multiplication, and define $\rk_G(A)=\dim_G \overline{\im \phi_G^A}$.

        The {\it group von Neumann algebra} $\VN(G)$ of $G$ is the algebra of left $G$-equivariant bounded operators on $\ell^2(G)$
        \[
        \VN(G):=\{\phi\in \mathcal{B}(\ell^2(G)):\phi(gv)=g\phi(v)\mbox{ for all }g\in G,v\in \ell^2(G)\}.
        \]
        The ring $\VN(G)$ is a finite von Neumann algebra. In particular, it is semihereditary according to \cref{thm: von_Neumann_alg_is_semihered}. Moreover, $\VN(G)$ satisfies the left Ore condition with respect to non-zero divisors (a result proved by S. K. Berberain in \cite{Berberian82}). We denote by $\UO(G)$ the left classical ring of fractions and called it the {\it ring of unbounded operators affiliated to} $G$. The ring $\UO(G)$ has nice properties, one of which is being a regular ring. We can extend the Sylvester matrix rank function $\rk_G$ to $\UO(G)$ as follows
        \[
        \rk_G(s^{-1}r):=\rk_G(r)=\dim_G(\overline{\ell^2(G)r}).
        \]
        This Sylvester matrix rank function is indeed the one associated to the L\"uck's dimension $\dim_{\scalebox{0.8}{$\UO(G)$}}$ \cite[Definition 8.28]{Luck02} (cf. \cite[Lemma 4.2.4]{DLopezAlvarezThesis})

        We record a well-known fact.
        
        \begin{lem}\label{lem: UOdim_equal_vNdim_for_elements_in_vN}
            Let $a_1,\ldots, a_l\in \VN(G)$, then $\dim_{\scalebox{0.8}{$\UO(G)$}}(\UO(G)a_1+\ldots+\UO(G)a_l)$ equals to $\dim_{\scalebox{0.8}{$\VN(G)$}}(\VN(G)a_1+\ldots+\VN(G)a_l)$.
        \end{lem}
        \begin{proof}
            Since $\VN(G)$ is a semihereditary ring, $\VN(G)a_1+\ldots+\VN(G)a_l$ is a projective left $\VN(G)$-module. Hence, since the left Ore localization is a flat right module \cite[Corollary 10.13]{GW04}, according to \cite[Corollary 3.59]{Rotman09} $\UO(G)\otimes_{\scalebox{0.8}{$\VN(G)$}}(\VN(G)a_1+\ldots+\VN(G)a_l)$ is isomorphic to $\UO(G)a_1+\ldots+\UO(G)a_l$ as left $\UO(G)$-modules.
        \end{proof}

        For a non-negative integer $k$, we define the {\it $k$th $L^2$-Betti number} of $G$ as 
        \[
        \dim_{\scalebox{0.8}{$\VN(G)$}}\left(\Tor_k^{\scalebox{0.8}{$\C[G]$}}(\VN(G),\C)\right)
        \]
        where $G$ acts trivially on $\C$. We can also replace $\VN(G)$ by $\UO(G)$ by flatness arguments. Not only that, since for any subgroup $H$ of $G$ it holds that (see \cite[Proposition 4.2.2]{DLopezAlvarezThesis})
        \begin{equation}\label{eqtn: preserve_dim_gps}
            \dim_{\scalebox{0.8}{$\UO(H)$}}(M)=\dim_{\scalebox{0.8}{$\UO(G)$}}(\UO(G)\otimes_{\scalebox{0.8}{$\UO(H)$}}M)
        \end{equation}
        for every left $\UO(H)$-module $M$, then applying Shapiro's lemma we get:
        
        \begin{lem}\label{lem: beta_H_as_dimG}
            Let $G$ be a group and $H$ a subgroup in $G$. Then for every non-negative integer $k$ it holds that $\dim_{\scalebox{0.8}{$\UO(G)$}}(\Tor_k^{\scalebox{0.8}{$\C[G]$}}(\UO(G),\C[G]\otimes_{\scalebox{0.8}{$\C[H]$}}\C))=\beta_k^{(2)}(H).$
        \end{lem}

    \subsection{Algebraic structure of discrete measured groupoids}\label{subsect: alg_struct_gpds}

        \begin{defn}
            A {\it groupoid} is a set $\gpd$ with a distinguished subset $\gpd^0$, range and source maps $r,s:\gpd\rightarrow \gpd^0$, a composition map $(\alpha,\beta)\mapsto \alpha\beta$ from $\{(\alpha,\beta)\in \gpd\times \gpd: s(\alpha)=r(\beta)\}$ to $\gpd$ and an inverse map $\gamma\mapsto \gamma^{-1}$ from $\gpd$ to $\gpd$ with the following properties:
		\begin{enumerate}[label=(\roman*)]
			\item $r(x)=x=s(x)$ for all $x\in \gpd^0$;
			\item $r(\gamma)\gamma=\gamma=\gamma s(\gamma)$ for all $\gamma\in \gpd$;
                \item $r(\gamma^{-1})=s(\gamma)$ and $s(\gamma^{-1})=r(\gamma)$ for all $\gamma\in \gpd$;
                \item $\gamma^{-1}\gamma=s(\gamma)$ and $\gamma\gamma^{-1}=r(\gamma)$ for all $\gamma\in \gpd$;
                \item $r(\alpha\beta)=r(\alpha)$ and $s(\alpha\beta)=s(\beta)$ whenever $s(\alpha)=r(\beta)$; and
                \item $(\alpha\beta)\gamma=\alpha(\beta\gamma)$ whenever $s(\alpha)=r(\beta)$ and $s(\beta)=r(\gamma)$.
		\end{enumerate}
	\end{defn}

        Let $(X,\mu)$ be a probability space with probability measure $\mu$ and let $G\times X\rightarrow X$ be a m.p. group action. We denote by $(X\rtimes G,\mu)$ or simply $X\rtimes G$ the {\it translation groupoid}, that is, the groupoid with set $X\times G$, distinguished subset $X\times \{1\}$, source map $s$ induced by $\pi_X$, range map $r$ defined via the action of $G$ on $X$, composition map given by $(x,g)(y,h):=(y,gh)$ provided that $(x,1)=(h\cdot y,1)$, and inverse map which sends $(x,g)$ to $(g\cdot x, g^{-1})$.

        Throughout sections \ref{sec: hom_features_gpd} and \ref{sec: new_vanishing_crit} the group $G$ will have an infinite index subgroup $H$ and the translation groupoid will be with respect to the probability space $X_{G,H}:=\prod_{gH\in G/H}[0,1]$ with measure the product of the Lebesgue measure on each factor and $G$ acting by left multiplication on the left cosets of $H$ in $G$.

        A {\it discrete measurable groupoid} $\gpd$ is a groupoid $\gpd$ equipped with the structure of a standard Borel space such that the composition and the inverse map are Borel and $s^{-1}(\{x\})$ is countable for all $x\in \gpd^0$.

        Given a probability measure $\mu$ on the distinguished subset $\gpd^0$ of a discrete measurable groupoid $\gpd$, for any measurable subset $A\subseteq \gpd$, the function $\gpd^0\rightarrow \mathbb{C}, x\mapsto \# (s^{-1}(\{x\})\cap A)$ is measurable and the measure $\mu_s$ on $\gpd$ defined by
        \[
        \mu_s(A)=\int_{\gpd^0}\#(s^{-1}(\{x\})\cap A)d\mu(x)
        \]
        is $\sigma$-finite (see \cite[Lemma 1.7]{SauerThesis}). The analogue statement holds if we replace $s$ by $r$.

        A discrete measurable groupoid $\gpd$ together with a probability measure on $\gpd^0$ satisfying $\mu_s=\mu_r$ is called {\it discrete measured groupoid}. 
        
        The translation groupoid $X\rtimes G$ is indeed a discrete measured groupoid provided that $G$ is countable.

        A {\it subgroupoid} $\Hgpd$ of a discrete measured groupoid $\gpd$ is a measurable subset of $\gpd$ having the groupoid structure with same distinguished subset $\gpd^0$ and same range, source, composition and inverse maps. Note that in this case $\Hgpd$ is automatically a discrete measured groupoid.         
        
        A discrete measured groupoid $(\gpd,\mu)$ is said to be {\it ergodic} if one of the following equivalent conditions hold.
	\begin{enumerate}[label=(\roman*)]
            \item Any measurable function $f:\gpd^0\rightarrow \mathbb{R}$ that is $\gpd$-invariant, i.e. $f\circ s=f\circ r$, is $\mu$-a.e. constant.
            \item For any Borel subset $A\subseteq \gpd^0$ of positive measure, the so-called {\it saturation} $A^{\gpd}:=\{x\in\gpd^0: \exists\gamma\in \gpd{\mbox{ with }s(\gamma)\in A,r(\gamma)=x}\}$ has full measure.
	\end{enumerate}
  
        Note that if the countable group $G$ has an infinite index subgroup $H$ then $X_{G,H}\rtimes G$ is ergodic. Indeed, if $A\subseteq \gpd^0$ has positive measure, then there are only finitely many components different from $[0,1]$ $\mu$-a.e. and the action of $G$ is transitive.

        By means of this ergodic notion it is possible to define a subgroup index analogue in the context of discrete measured groupoids. In fact, first observe that every groupoid induces an equivalence relation on its distinguished subset:
        \[
        x\sim y\Leftrightarrow \exists\gamma\in \gpd \mbox{ such that } s(\gamma)=x \mbox{ and } r(\gamma)=y, \mbox{ for } x,y\in \gpd^0.
        \]
        Similarly, if $\Hgpd$ is a subgroupoid of $\gpd$ then it induces a subrelation on the same distinguished subset. In particular, it holds that the $\Hgpd$-orbit $x^{\Hgpd}$ is contained in the $\gpd$-orbit $x^{\gpd}$ for every $x\in \gpd^0$. So if we let $J(x)$ denote the quotient $x^{\gpd}/x^{\Hgpd}$, then $\#J(x)$ is constant $\mu$-a.e. (cf \cite[Lemma 1.1]{FeldShutZimm}). This value will be called the {\it index} of $\Hgpd$ in $\gpd$. Note that we also allow infinite indices.

        The next result shall play a central role in our discussion. We recall that a {\it measure isomorphism} $f:(X,\mu_X)\rightarrow (Y,\mu_Y)$ is a measure preserving Borel map with the property that there are Borel subsets $A\subseteq X$ and $B\subseteq Y$ with $\mu_X(X\setminus A)=\mu_Y(Y\setminus B)=0$ such that $f\vrule_A$ is a Borel isomorphism, i.e. a bijective Borel map.
        
        \begin{lem}{\cite[Lemma 3.8]{SauerThom}}\label{lem: section_maps_for_gpds}
            Let $\gpd$ be an ergodic discrete measured groupoid. Let $E,F\subseteq \gpd^0$ be Borel subsets with the same measure. Then there is a Borel map $\varphi:\gpd^0\rightarrow \gpd$ such that 
			\begin{enumerate}[label=(\roman*)]
                    \item $s(\varphi(x))=x$ $\mu$-a.e. and $x\mapsto r(\varphi(x))$ is $\mu$-a.e. injective for $x\in \gpd^0$; and
                    \item $r(\varphi(E))\subseteq F$ and $r\circ \varphi:E\rightarrow F$ is a measure isomorphism.
			\end{enumerate}
        \end{lem}

        Let $n$ be a positive integer; for $i=1,\ldots, n$ we set the following subsets of $X_{G,H}\rtimes G$
        \[
        Y_i=\left(\left[\frac{i-1}{n},\frac{i}{n}\right]\times \prod_{\substack{gH\in G/H\\ gH\neq H}}[0,1]\right)\times\{1\}.
        \]
        We can now define as in \cite[Lemma 6.2]{PetersonThom_Freiheit} the subgroupoid $\gpd_n\subseteq X_{G,H}\rtimes G$ which consists of those elements in $\gpd$ that preserve the partition, that is
        \[
        \gpd_n=\{\gamma\in X_{G,H}\rtimes G: \mbox{ both } s(\gamma),r(\gamma)\in Y_i \mbox{ for some }i\in \{1,\ldots, n\}\}.
        \]
        From \cref{lem: section_maps_for_gpds}, it follows that $\gpd_n$ has index $n$ in $\gpd$. Hence, in contrast to the rigidity of groups, we can create finite index subgroupoids of arbitrarily large index.

        Moreover, discrete measured groupoids carry a natural algebraic structure. Before introducing the group algebra analogue, we shall define two associated functions. For a map $\phi:\gpd\rightarrow \mathbb{C}$ and $x\in \gpd^0$ we set
	\begin{align*}
            &S(\phi)(x):=\#\{\gamma\in \gpd: \phi(\gamma)\neq 0, s(\gamma)=x\}\in \mathbb{Z}_{\geq 0}\cup\{\infty\};\\
            &R(\phi)(x):=\#\{\gamma\in \gpd: \phi(\gamma)\neq 0, r(\gamma)=x\}\in \mathbb{Z}_{\geq 0}\cup\{\infty\}.
	\end{align*} 
        As it is customary, the set of complex-valued, measurable, essentially bounded functions up to zero measure on $\gpd$ with respect to $\mu$ is denoted by $L^{\infty}(\gpd)$.
        
	\begin{defn}
		The {\it groupoid ring} $\C[\gpd]$ of a discrete measured groupoid $\gpd$ is defined as
            \[
            \C[\gpd]=\{\phi\in L^{\infty}(\gpd): S(\phi),R(\phi)\in L^{\infty}(\gpd^0)\}.
            \]
	\end{defn}

        The set $\C[\gpd]$ is a ring with involution where addition is pointwise addition in $L^{\infty}(\gpd)$, multiplication is given by the convolution product
        \[
        (\phi\eta)(\gamma)=\sum_{\substack{\gamma_1,\gamma_2\in\gpd \\ \gamma_2\gamma_1=\gamma}}\phi(\gamma_1)\eta(\gamma_2)
        \]
        and the involution is defined by $(\phi^{\ast})(\gamma)=\overline{\phi(\gamma^{-1})}$ (cf. \cite[Lemma 1.20]{SauerThesis}). Observe that this ring contains $L^{\infty}(\gpd^0)$ as a subring. Indeed, $L^{\infty}(\gpd^0)\hookrightarrow \C[\gpd]$ via $f\mapsto f\chi_{\gpd^0}$ which respects the convolution product. Not only that, $L^{\infty}(\gpd^0)$ can be equipped with a left $\C[\gpd]$-module structure. Concretely, the {\it augmentation homomorphism} $\epsilon:\C[\gpd]\rightarrow L^{\infty}(\gpd^0)$ is defined as
        \[
        \epsilon(\phi)(x):=\sum_{\gamma\in s^{-1}(\{x\})}\phi(\gamma).
        \]
        Now set $\eta\cdot f:=\epsilon(\eta f)$ for $\eta\in \C[\gpd],f\in L^{\infty}(\gpd^0)$ where $\eta f$ is the convolution product in $\C[\gpd]$ (see \cite[Lemma 1.22]{SauerThesis}). 
        
        We record a useful result for future reference that also gives us a better insight of the form of the objects in the groupoid ring.

        \begin{lem}{\cite[Lemma 3.3]{SauerGroupoids}}\label{lem: gpd_ring_are_finite_sums}
            As a left $\essbdd$-module the groupoid ring $\C[\gpd]$ is generated by the characteristic functions $\chi_E$ of Borel subsets $E\subseteq \gpd$ with the property that $s\vrule_E$ and $r\vrule_E$ are injective.
        \end{lem}

        There is also an operator algebra structure involved. Indeed, from the proof of \cite[Theorem 1.50]{SauerThesis} it follows that $\C[\gpd]$ acts on $L^2(\gpd,\mu)$ by left and right. This way taking the right regular action, we consider $\C[\gpd]$ as a subalgebra of the bounded operators $\mathcal{B}(L^2(\gpd,\mu))$. This leads to define the {\it groupoid von Neumann algebra} $\VN(\gpd)$ of $\gpd$ as the algebra of left $\gpd$-equivariant bounded operators on $L^2(\gpd,\mu)$
        \[
        \VN(\gpd):=\{\Psi\in \mathcal{B}(L^2(\gpd,\mu)): \Psi(\phi\cdot \eta)=\phi\cdot\Psi(\eta) \mbox{ for all }\eta\in L^2(\gpd,\mu), \phi\in \C[\gpd]\}.
        \]
        \vspace*{-7mm}
        \begin{rem}
            This definition coincides with \cite[Definition 1.51]{SauerThesis}. Indeed, the proof of \cite[Theorem 2.24]{Kamm_l2invt} can be carried out line by line by means of \cite[theorem 1 on p.80]{Dixmier} and \cite[Theorem 2.19]{Kamm_l2invt}.
        \end{rem}

        The von Neumann algebra $\VN(\gpd)$ has a finite trace $\trace_{\scalebox{0.8}{$\VN(\gpd)$}}$, induced by the invariant measure $\mu$
        \[
        \trace_{\scalebox{0.8}{$\VN(\gpd)$}}(\phi)=\int_{\gpd^0}\phi(x)d\mu(x)
        \]
        for $\phi\in \C[\gpd]\subseteq \VN(\gpd)$. In more detail, $\C[\gpd]$ has an inner product given by
        \[
        \langle \phi,\eta\rangle=\int_{\gpd}\phi(\gamma)\overline{\eta(\gamma)}d\mu(\gamma);
        \]
        and the trace function on $\VN(\gpd)$ is also computed as $\langle \psi(1),1\rangle$ for $\psi\in \VN(\gpd)$ and $1$ the identity function, where the inner product is taken in the Hilbert space completion of $(\C[\gpd],\langle , \rangle)$ (see \cite[Theorem 1.46 \& Theorem 1.50]{SauerThesis} for further details).
        
        This trace function endows a Sylvester matrix rank function on $\VN(\gpd)$ as usual (cf. \cite[Section 2.4]{JiangLi_Amenable}). This induces a SMRF on finitely presented left $\VN(\gpd)$-modules, which we extend to a BSMRF over arbitrary modules according to \cref{thm: SMRF_to_BSMRF}. Thus we can now define our main objects of interest.
	\begin{defn}
            Let $\gpd$ be a discrete measured groupoid, for every non-negative integer $k$ we define the {\it $k$th $L^2$-Betti number} of $\gpd$ as
            \[
            \dim_{\scalebox{0.8}{$\VN(\gpd)$}}\left(\Tor_k^{\scalebox{0.8}{$\C[\gpd]$}}(\VN(\gpd),L^{\infty}(\gpd^0))\right).
            \]
	\end{defn}

        A ring $S$ is $G$-{\it graded} if $S = \bigoplus_{g\in G} S_g$ as an additive group, where $S_g$ is an additive subgroup for every $g\in G$, and $S_g S_h \subseteq S_{gh}$ for all $g,h \in G$. If $S_g$ contains an invertible element $u_g$ for each $g\in G$, then we say that $S$ is a \textit{crossed product} of $S_e$ and $G$ and we shall denote it by $S=S_e * G$. Note that the usual group ring $R[G]$ of a group $G$ with $R$ a ring is an example of a crossed product.

        In the case of the translation groupoid there will be another ring that comes into play. The {\it crossed product ring} $L^{\infty}(\gpd^0)\ast G$ is the free left $L^{\infty}(\gpd^0)$-module with basis $G$, where the multiplication is determined by the rule $g\cdot f(\_)=f(g^{-1}\cdot\_)g$ for $g\in G$ and $f\in L^{\infty}(\gpd^0)$.
	\begin{rem}
            The crossed product ring embeds into $\C[X\rtimes G]$ via $\sum f_gg\mapsto f\in L^{\infty}(\gpd)$ with $f(x,g^{-1})=f_g(x)$. Furthermore, $\phi\in \C[X\rtimes G]$ is in the image if and only if there is a finite subset $F\subseteq G$ such that $g\in G\setminus F$ implies that $\phi(x,g)=0$ $\mu$-a.e. (this follows from the proof of \cite[Lemma 1.26]{SauerThesis}).
	\end{rem}

        There is a natural relation between the group algebra of $G$ and that of its translation groupoid (the proof of \cite[Corollary 1.54]{SauerThesis} remains valid in this setting).

        \begin{prop}\label{prop: VN_gpd_is_trace_preserving}
            The $*$-homomorphism $\C[G]\rightarrow \C[X\rtimes G]$ extends to a trace-preserving $*$-homomorphism $\VN(G)\rightarrow \VN(X\rtimes G)$.
        \end{prop}

\section{Homological Algebra and Dimension Theory} \label{sec: hom_alg_dim_theory}
		
    In this section we follow the approach developed by Sauer in \cite[Section 4]{SauerGroupoids}. Mainly, we extend its work to exact SMRF and give a strengthen version of \cite[Theorem 4.11]{SauerGroupoids} (see \cref{thm: dim_isomph_changing_scalars}). We begin with a convenient lemma.

    \begin{lem}\label{lem: dim_directed_union}
        Let $R$ be a ring with SMRF $\dim_R$ and $M$ a left $R$-module. If $M=\cup_{i\in I}M_i$ is a directed union of left $R$-submodules, then
        \[
        \dim_R(M)\leq \sup_{i\in I}\dim_R(M_i).
        \]
    \end{lem}
    
    \begin{proof}
        First recall that by \cref{defn: BSMRF} $\dim_R(M)=\sup_{M'}\dim_R(M'$\textbar$ M)$ where $M'$ ranges over the finitely generated left $R$-submodules of $M$. Choose one such $M'$; then, since it is finitely generated, there is some $i_0\in I$ satisfying $M'\subseteq M_{i_0}$. Thus according to \cref{lem: continuity_dim}
        \[
        \dim_R(M'\mbox{\textbar} M)\leq \dim_R(M'\mbox{\textbar} M_{i_0})\leq \dim_R (M_{i_0})\leq \sup_{i\in I}\dim_R (M_i).
        \]
        Since $M'$ was arbitrary, this proves the lemma.
    \end{proof}
    
    We can use the above result to show that bimodules with vanishing dimension still have dimension zero when taking tensor product on the right hand side.
    
    \begin{lem}\label{lem: dim_zero_extends_by_right}
        Let $T$ be a ring with SMRF $\dim_T$ and $L$ a $T$-$R$-bimodule with $\dim_T(L)=0$. Then for every left $R$-module $M$ it holds that
        \[
        \dim_T (L\otimes_RM)=0.
        \]
    \end{lem} 
    
    \begin{proof}
        We write $M$ as a directed union of its finitely generated left $R$-submodules, say $M=\cup_{i\in I}M_i$. Now since tensor product and direct limit commute we get
        \[
        \dim_T (L\otimes_R M)=\dim_T \left(\cup_{i\in I}(L\otimes_R M_i)\right).
        \]
        According to \cref{lem: dim_directed_union}
        \[
        \dim_T (L\otimes_R M)\leq \sup_{i\in I}\dim_T (L\otimes_R M_i).
        \]
        Moreover, $M_i$ is finitely generated for all $i\in I$. Hence, for each $i$ there must be some $n_i\in \mathbb{N}$ such that $R^{n_i}$ surjects onto $M_i$. Thus, since tensor product is right exact, by \cref{defn: BSMRF} it holds
        \begin{align*}
            \dim_T (L\otimes_R M)\leq \sup_{i\in I}\dim_T (L\otimes_R M_i)& \leq \sup_{i\in I}\dim_T (L\otimes_R R^{n_i}) \\
             &= \sup_{i\in I} n_i\cdot\dim_T (L)=0.
        \end{align*}
    \end{proof}
    
    \begin{rem}\label{rem: vanish_proj_resol}
        With the above notation, observe that if $P_{\bullet}:\ldots\rightarrow P_1\rightarrow P_0\rightarrow M\rightarrow 0$ is a projective resolution for the left $R$-module $M$, then by \cref{lem: dim_zero_extends_by_right} it holds that $\dim_T (L\otimes_R P_i)=0$. In particular, if $\dim_T$ is exact then $\dim_T(\Tor_k^R(L,M))=0$ for every non-negative integer $k$.
    \end{rem}
    
    In this article we shall be interested in computing dimensions of certain modules. Recall from \cref{defn: BSMRF} that isomorphic modules have same dimension. However, the notion of isomorphism is rigid in the sense that there might be non-isomorphic modules with same dimension. So it is reasonable to work with the following weaker notion.
    
    \begin{defn}\label{defn: dim_isomph}
        Let $R$ be a ring with SMRF $\dim_R$. An $R$-homomorphism $f:M\rightarrow N$ between left $R$-modules $M,N$ is called a {\it left $\dim_R$-isomorphism} if $\dim_R(\ker f)=\dim_R(\coker f)=0$. If there exists a left $\dim_R$-isomorphism between two left modules, then we say that the modules are {\it $\dim_R$-isomorphic}.
    \end{defn}
    
    Note that a similar notion can be defined for right $R$-modules. When working with dimensions, one should think of dimension-isomorphic modules as usual isomorphic modules. For the reader's convenience, we write down a proof of the next obvious fact.
    
    \begin{lem}\label{lem: comp_dim_isomph_is_dim_isomph}
        Let $R$ be a ring with exact SMRF $\dim_R$ and let $f:M_1\rightarrow M_2$ and $g: M_2\rightarrow M_3$ be both left $\dim_R$-isomorphisms. Then $g\circ f:M_1\rightarrow M_3$ is a left $\dim_R$-isomorphism.
    \end{lem}
    
    \begin{proof}
        Consider the next two exact sequences
        \[
        M_2/f(M_1)\rightarrow M_3/g\circ f(M_1)\rightarrow M_3/g(M_2)\rightarrow 0
        \]
        and
        \[
        0\rightarrow\ker f\rightarrow \ker g\circ f \rightarrow f(M_1)\cap\ker g\rightarrow 0.
        \]
        Note that since the dimension is exact, $\dim_R (f(M_1)\cap \ker g)\leq \dim_R (\ker g)$. So then taking dimensions on both sequences yield to the desired result by \cref{defn: BSMRF} along with \cref{lem: continuity_dim}.
    \end{proof}
    
    We now show that the functor $\Tor$ is left dimension isomorphic invariant after changing left dimension isomorphic bimodules on the left hand side.
    
    \begin{lem}\label{lem: dim_isomph_tensor_right}
        Let $T$ be a ring with exact SMRF $\dim_T$ and $L_1,L_2$ $T$-$R$-bimodules. If $\varphi:L_1\rightarrow L_2$ is a $T$-$R$-homomorphism which is a left $\dim_T$-isomorphism, then for every non-negative integer $k$ the induced map
        \[
        \Tor_k^R(L_1,M)\longrightarrow\Tor_k^R(L_2,M)
        \]
        is a left $\dim_T$-isomorphism for every left $R$-module $M$.
    \end{lem}
    
    \begin{proof}
        We start by considering the next two short exact sequences of $T$-$R$-bimodules
        \[
        0\rightarrow\ker\varphi\rightarrow L_1\rightarrow \im\varphi\rightarrow 0
        \]
        and
        \[
        0\rightarrow\im\varphi\rightarrow L_2\rightarrow \coker\varphi\rightarrow 0.
        \]
        Taking tensor product on the first one yields to the following long exact sequence
        \begin{align*}
            \ldots&\rightarrow \Tor_1^R(\ker \varphi, M)\rightarrow \Tor_1^R(L_1,M)\rightarrow\Tor_1^R(\im \varphi, M)\\
            &\rightarrow \ker\varphi\otimes_R M\rightarrow L_1\otimes_R M\rightarrow \im \varphi \otimes_R M\rightarrow 0.
        \end{align*}
        However, by assumption $\dim_T(\ker \varphi) =0$; hence by \cref{rem: vanish_proj_resol}, since $\dim_T$ is also exact, $\dim_T (\Tor_k^R(\ker \varphi, M))=0$ for all $k\geq 0$. So
        \[
        \Tor_k^R(L_1,M)\rightarrow\Tor_k^R(\im \varphi,M)
        \]
        is a left $\dim_T$-isomorphism for all $k\geq 0$. Similarly, from the second short exact sequence we obtain that
        \[
        \Tor_k^R(\im \varphi,M)\rightarrow\Tor_k^R(L_2,M)
        \]
        is a left $\dim_T$-isomorphism for all $k\geq 0$. Composing both maps we get the desired left $\dim_T$-isomorphism according to \cref{lem: comp_dim_isomph_is_dim_isomph}.
    \end{proof}
    
    The above results show how the functor $\Tor$ behaves when taking tensor product on the right. On the other hand, let us now consider what happens when the tensor product is taking on the left hand side. Before so, we introduce a definition to ease the notation.
    
    \begin{defn}\label{defn: dim_compatibility}
        Let $R,T$ be two rings with SMRF $\dim_R$ and $\dim_T$, respectively. A $T$-$R$-bimodule $L$ is called {\it $\dim_T$-$\dim_R$-compatible} if for every left $R$-module $M$ with $\dim_R(M)=0$ it holds that $\dim_T(L\otimes_RM)=0$.
    \end{defn}
    
    \begin{rem}\label{rem: dim_compatibility_behaviour}
        This is a natural extension of the concept of {\it dimension-compatible} defined in \cite[Definition 4.6]{SauerGroupoids} for bimodules of finite von Neumann algebras. In the same fashion as in \cite[Lemma 4.7]{SauerGroupoids}, this compatibility notion behaves as expected with tensor product, quotients and direct summands.
    \end{rem}

    Unlike in \cref{lem: dim_zero_extends_by_right}, this compatibility notion turns out to be the key point to preserve the vanishing dimension of a left module when we tensor it with a bimodule.
    
    \begin{lem}\label{lem: dim_zero_extends_by_left}
        Let $R\subseteq S$ and $T$ be rings, $\dim_T$ an exact SMRF for $T$, $\dim_R$ a SMRF for $R$ and $M$ a left $S$-module with $\dim_R (M)=0$. Suppose that $S$ is a flat right $R$-module which is $\dim_R$-$\dim_R$-compatible and $L$ is a $T$-$S$-bimodule which is $\dim_T$-$\dim_R$-compatible and flat as right $R$-module. Then for every non-negative integer $k$ it holds that
        \[
        \dim_T(\Tor_k^S(L,M))=0.
        \]
    \end{lem}
    
    \begin{proof}
        First consider the natural short exact sequence of $T$-$S$-bimodules
        \[
        0\rightarrow \ker \pi_1\rightarrow L\otimes_R S\xrightarrow{\pi_1} L\rightarrow 0.
        \]
        Now we tensor by right with $M$; hence we get
        \begin{align*}
            \ldots&\rightarrow\Tor_2^S(L\otimes_R S,M)\rightarrow \Tor_2^S(L,M)\rightarrow \Tor_1^S(\ker\pi_1,M)\\
            &\rightarrow\Tor_1^S(L\otimes_R S,M)\rightarrow \Tor_1^S(L,M)\rightarrow \ker\pi_1\otimes_S M\\
            &\rightarrow L\otimes_R S\otimes_S M\rightarrow L\otimes_S M\rightarrow 0.
        \end{align*}
        Note that $L\otimes_R S$ is flat as right $S$-module. So for every $k\geq 2$
        \[
        \Tor_k^S(L,M)\cong \Tor_{k-1}^S(\ker \pi_1, M).
        \]
        Moreover, $\ker \pi_1$ is a direct summand of $L\otimes_R S$ as $T$-$R$-bimodule because of the section map $l\mapsto l\otimes 1$ for $l\in L$. In particular, since $L\otimes_R S$ is $\dim_T$-$\dim_R$-compatible so is $\ker \pi_1$. Hence
        \[
        0=\dim_T (\ker \pi_1\otimes_R M)\geq \dim_T(\ker \pi_1\otimes_S M).
        \]
        Therefore, by exactness of $\dim_T$ we get that
        \[
        \dim_T(\Tor_1^S(L,M))=0.
        \]
        Not only that, $L\otimes_R S$ is flat as right $R$-module; and then since $L$ is flat as right $R$-module, so is $\ker \pi_1$. To sum up, $\ker \pi_1$ satisfies all the hypothesis imposed on $L$. Thus, to show the vanishing of the dimension of higher order terms we proceed by dimension shifting arguing with $\ker \pi_1$ instead of $L$.
    \end{proof}
    
    We now show that the functor $\Tor$ behaves also well after changing dimension isomorphic left modules on the right hand side in the following circumstances.
    
    \begin{lem}\label{lem: dim_isomph_tensor_left}
        Let $R\subseteq S$ and $T$ be rings, $\dim_T$ an exact SMRF for $T$ and $\dim_R$ a SMRF for $R$. Suppose that $S$ is a flat right $R$-module which is $\dim_R$-$\dim_R$-compatible and $L$ is a $T$-$S$-bimodule which is $\dim_T$-$\dim_R$-compatible and flat as right $R$-module. If $\varphi: M_1\rightarrow M_2$ is a left $S$-homomorphism  which is a left $\dim_R$-isomorphism, then for every non-negative integer $k$ the induced map
        \[
        \Tor_k^S(L,M_1)\longrightarrow\Tor_k^S(L,M_2)
        \]
        is a left $\dim_T$-isomorphism.
    \end{lem}
    
    \begin{proof}
        Follows directly from \cref{lem: dim_zero_extends_by_left} replicating the proof of \cref{lem: dim_isomph_tensor_right}.
    \end{proof}
    
    To end this section we state sufficient conditions in order to change the ring under we are taking tensor product.
    
    \begin{thm}\label{thm: dim_isomph_changing_scalars}
        Let $R\subseteq S_1\subseteq S_2$ and $T$ be rings and let $\dim_T$ and $\dim_R$ be exact SMRF for $T$ and $R$, respectively. Suppose that $S_1\subseteq S_2$ is a left $\dim_R$-isomorphism, $S_2$ is a flat right $R$-module which is $\dim_R$-$\dim_R$-compatible and $L$ is a $T$-$S_2$-bimodule which is $\dim_T$-$\dim_R$-compatible and flat as right $R$-module. Then for every left $S_2$-module $M$ the homomorphism $S_2\otimes_{S_1}M\rightarrow M$ induces a left $\dim_T$-isomorphism
        \[
        \Tor_k^{S_1}(L,M)\longrightarrow\Tor_k^{S_2}(L,M)
        \]
        for every non-negative integer $k$.
    \end{thm}
    
    \begin{proof}
        According to \cref{lem: dim_isomph_tensor_right}, since $S_1\subseteq S_2$ is a left $\dim_R$-isomorphism and $\dim_R$ is exact, the induced map
        \[
        \Tor_k^{S_1}(S_1,M)\longrightarrow\Tor_k^{S_1}(S_2,M)
        \]
        is a left $\dim_R$-isomorphism for every $k\geq 0$. Not only that, these maps are left $S_2$-homomorphism for $k>0$ and for $k=0$ the left inverse map, that is the natural projection $S_2\otimes_{S_1}M\rightarrow M$, is a left $S_2$-homomorphism which has to be a left $\dim_R$-isomorphism too.

        Moreover, let $P_{\bullet}:\ldots\rightarrow P_2\rightarrow P_1\rightarrow P_0\rightarrow M\rightarrow 0$ be a projective resolution for the left $S_1$-module $M$. The K\"uneth spectral sequence, applied to $L$ and the complex $S_2\otimes_{S_1}P_{\bullet}$ (see \cite[Theorem 5.6.4]{Weibel94}), has $E^2$-term
        \[
        E^2_{pq}=\Tor_p^{S_2}(L,H_q(S_2\otimes_{S_1}P_{\bullet}))=\Tor_p^{S_2}(L,\Tor_q^{S_1}(S_2,M))
        \]
        and converges to
        \[
        H_{p+q}(L\otimes_{S_2}(S_2\otimes_{S_1}P_{\bullet}))=\Tor_{p+q}^{S_1}(L,M).
        \]
        Now we apply \cref{lem: dim_isomph_tensor_left} for $R,S_2$ and $T$ to get left $\dim_T$-isomorphisms
        \begin{equation}\label{eqtn: spect_seq_1}
            \Tor_p^{S_2}(L,\Tor_q^{S_1}(S_1,M))\rightarrow \Tor_p^{S_2}\left(L,\Tor_q^{S_1}(S_2,M)\right),
        \end{equation}
        if $q>0$, and
        \begin{equation}\label{eqtn: spect_seq_2}
            \Tor_p^{S_2}\left(L,S_2\otimes_{S_1}M\right)\rightarrow \Tor_p^{S_2}(L,M)
        \end{equation}
        for $q=0$. We recall that $E^{r+1}_{pq}\cong \ker d^r_{pq}/\im d^r_{p+r,q-r+1}$ where $d^r_{pq}$ maps $E^r_{pq}$ to $E^r_{p-r,q+r-1}$ and $r\geq 2$. Hence if $q>0$, $\dim_T (E_{pq}^r)=0$ by (\ref{eqtn: spect_seq_1}); and so, $E_{pq}^{\infty}$ is a module with vanishing dimension. However, when $q=0$ the maps $d^r_{pq}$ have zero-dimensional target modules. Thus for $q=0$, $E^{r+1}_{pq}$ is $\dim_T$-isomorphic to $E_{pq}^r$ since $q-r+1<0$. Note that, for $r>p$ it holds that $E_{p0}^r=E_{p0}^{\infty}$. Therefore, $\Tor_p^{S_1}(L,M)$ is $\dim_T$-isomorphic to $E_{p0}^{\infty}$ which is $\dim_T$-isomorphic to $E^2_{p0}$ according to \cref{lem: comp_dim_isomph_is_dim_isomph}; and hence $\dim_T$-isomorphic to $\Tor_p^{S_2}(L,M)$ by (\ref{eqtn: spect_seq_2}).
    \end{proof}

    \begin{rem}
        The proofs of both \cref{prop: shapiro_dim_for_gpds} and \cref{prop: vanish_dim_mod_criteria_for_gpds} need the existence of this left dim-isomorphism, which has been the truly motivation to further extend \cite[Theorem 4.11]{SauerGroupoids}.
    \end{rem}

\section{Homological features of the translation groupoid}\label{sec: hom_features_gpd}

    As aforementioned in \cref{subsect: alg_struct_gpds}, whenever a countable group $G$ acts m.p. on a probability space $X$, we can define the discrete measured groupoid $X\rtimes G$. From now on, we shall denote by $\gpd$ the translation groupoid $X_{G,H}\rtimes G$, by $\Hgpd$ the translation subgroupoid $X_{G,H}\rtimes H$ and by $\gpd_n$ the intermediate subgroupoids of \cref{subsect: alg_struct_gpds} where $H$ is an infinite index subgroup of $G$ and $X_{G,H}$ stands for $\prod_{gH\in G/H}[0,1]$ with $G$ acting by left multiplication.

    The next result shows that the groupoid algebras still share some nice homological properties. First recall that $\gpd_n=\{(x,g)\in\gpd:\mbox{ both } s(x,g),r(x,g)\in Y_i \mbox{ for some }i\in\{1,\ldots,n\}\}$ where 
    \[
    Y_i=\left(\left[\frac{i-1}{n},\frac{i}{n}\right]\times \prod_{\substack{gH\in G/H\\ gH\neq H}}[0,1]\right)\times\{1\}.
    \]
    We review the argument that shows the finite index property since in this case we have much more than that. Suppose $(x,g)\in \gpd$ satisfies that $s(x,g)\in Y_j$ and $r(x,g)\in Y_i$ for some $i,j\in \{1,\ldots, n\}$. According to \cref{lem: section_maps_for_gpds} the Borel map $\varphi_{j,i}$ associates to $s(x,g)$ an element $\varphi_{j,i}(s(x,g))$ with range in $Y_i$. Moreover, $\mu$-a.e. $\gamma_{x,g}:=(x,g)\varphi_{j,i}(s(x,g))^{-1}$ is well defined, and so it belongs to $\gpd_n$. Thus $\mu$-a.e. we can write elements of $\gpd$ as a composition of an element in $\gpd_n$ and another element with same source, namely $(x,g)=\gamma_{x,g}\cdot \varphi_{j,i}(s(x,g))$. 
    
    Furthermore, $\mu$-a.e. given $\gamma\in \gpd_n$ and some $j\in \{1,\ldots, n\}$, there is a unique element $(x,g)\in \gpd$ such that $\gamma=\gamma_{x,g}$. Indeed, assume that $s(\gamma),r(\gamma)\in Y_i$ for some $i\in \{1,\ldots, n\}$. Since $r\circ \varphi_{j,i}:Y_j\rightarrow Y_i$ is $\mu$-a.e. surjective, then $\mu$-a.e. there is some $\beta\in \gpd$ satisfying that $s(\beta)\in Y_j$ and $r\circ\varphi_{j,i}(s(\beta))=s(\gamma)$. Thus $\mu$-a.e. $\gamma\varphi_{j,i}(s(\beta))$ is a well defined element in $\gpd$ and satisfies the desired property since $s\circ \varphi_{j,i}$ is $\mu$-a.e. injective. On the other hand, the uniqueness part goes as follows. If $(x_1,g_1)$ and $(x_2,g_2)$ are both candidates for $\gamma$ and $j$, then
    \[
    (x_1,g_1)\varphi_{j,i}(s(x_1,g_1))^{-1}=\gamma=(x_2,g_2)\varphi_{j,i}(s(x_2,g_2))^{-1}.
    \]
    However, $r\circ \varphi_{j,i}(s(x_1,g_1))=s(\gamma)=r\circ \varphi_{j,i}(s(x_2,g_2))$ and $r\circ\varphi_{j,i}$ is $\mu$-a.e. injective. Hence, $s(x_1,g_1)=s(x_2,g_2)$ $\mu$-a.e. and by right cancellation in the above equality, we get that $(x_1,g_1)=(x_2,g_2)$ $\mu$-a.e. as claimed.

    Not only that, this assignment is done in a measurable way. In fact, for each $i\in \{1,\ldots,n\}$ since $s\circ \varphi_{j,i}$ is $\mu$-a.e. injective and $r\circ \varphi_{j,i}$ is a measure isomorphism, there are Borel subsets $Y_{j,i}^1\subseteq Y_j$ and $Y_{j,i}^2\subseteq Y_i$ with $\mu(Y_j\setminus Y_{j,i}^1)=\mu(Y_i\setminus Y_{j,i}^2)=0$ such that $s\circ \varphi_{j,i}\vrule_{Y_{j,i}^1}$ is injective Borel and $r\circ \varphi_{i,j}:Y_{j,i}^1\rightarrow Y_{j,i}^2$ is bijective Borel. Thus, for $\gamma \in s^{-1}(Y_{j,i}^2)\cap \gpd_n$, $s(\beta)$ is obtained from the composition of Borel maps given by $(r\circ \varphi_{j,i})^{-1}\circ s(\gamma)$; and hence the map $f_{j,i}:s^{-1}(Y_{j,i}^2)\cap \gpd_n\rightarrow \gpd$ defined by $\gamma \mapsto \gamma \varphi_{j,i}(s(\beta))$ is Borel. This induces an injective Borel map $F_j$ on the Borel subset $\gpd_n^j:=\uplus_{i=1}^n s^{-1}(Y_{j,i}^2)\cap \gpd_n$ of $\gpd_n$ which is of full measure.

    Let us briefly give an interpretation of the map $F_j$. First, set $A_j:=\uplus_{i=1}^n\varphi_{j,i}(Y_{j,i}^1)$ which is Borel by injectivity (see \cite[Corollary 15.2]{Kech_setth}). We have shown that $\mu$-a.e. for each element in $\gpd_n$ there is a unique element of $\gpd$ that can be written as a composition of the element in $\gpd_n$ followed by an element in $A_j$. Thus $F_j$ assigns such composition in the full measure Borel subset $\gpd_n^j$ of $\gpd_n$. 

    We now focus on the algebraic structure. Note that $\C[\gpd_n]$ is naturally a subring of $\C[\gpd]$.
       
    \begin{lem}\label{lem: gpd_ring_is_free_gpd_n_module}
	$\C[\gpd]$ is free as right $\C[\gpd_n]$-module.
    \end{lem}
    
    \begin{proof}
        We claim that $\{\chi_{A_1},\ldots,\chi_{A_n}\}$ is a basis of $\C[\gpd]$ as right $\C[\gpd_n]$-module. Indeed, given $\phi\in \C[\gpd]$, choose some $(x,g)\in \gpd$ and suppose that $s(x,g)\in Y_j$ and $r(x,g)\in Y_i$. Then we know that $\mu$-a.e. $(x,g)=\gamma_{x,g}\varphi_{j,i}(s(x,g))$, and hence $(x,g)=F_j(\gamma_{x,g})$ $\mu$-a.e. In other words, $\phi(x,g)=\phi(F_j(\gamma_{x,g}))$ $\mu$-a.e. Thus if we define $\phi_j\in \C[\gpd_n]$ with support in $\gpd_n^j$ as $\phi\circ F_j$, for $j = 1,\ldots, n$, we get that
        \[
        \phi=\chi_{A_1}\cdot\phi_{1}+\ldots+\chi_{A_n}\cdot \phi_{n}
        \]
        as $L^{\infty}(\gpd)$ functions. Finally, they are independent over $\gpd$ because the decomposition of elements of $\gpd$ as a product of elements in $\gpd_n$ followed by elements in $A_j$ for some $j\in \{1,\ldots, n\}$ can be done $\mu$-a.e.
    \end{proof}

    \begin{rem}
        Similarly, it can be shown that $\C[\gpd]$ is a free left $\C[\gpd_n]$-module taking the dual decomposition.
    \end{rem}
  
    One particular interest of the above statement is that it enables us to apply Shapiro's lemma in the context of groupoid algebras.  This will be crucial to use the tools of homological algebra.

    Not only that, we can also understand the relation between $\VN(\gpd)$ and $\VN(\gpd_n)$. On the one side, there is a unitary map between the Hilbert spaces $\oplus_{j=1}^nL^2(\gpd_n)$ and $L^2(\gpd)$. Indeed, given $(\phi_1,\ldots,\phi_n)\in \oplus_{j=1}^nL^2(\gpd_n)$ we set $\phi:=\chi_{A_1}\cdot \phi_{1}+\ldots+\chi_{A_n}\cdot \phi_{n}$ and denote this map by $u$. Similarly, we have a unitary map $u^*$ from $L^2(\gpd)$ to $\oplus_{j=1}^nL^2(\gpd_n)$ that sends $\phi\in L^2(\gpd)$ to $(\phi_1,\ldots,\phi_n)$ where $\phi_j$ is a function with support in $\gpd_n^j$ defined as $\phi\circ F_j$. In particular, given $\Psi\in \VN(\gpd)$ we can associate an element in $\Mat_n(\VN(\gpd_n))$, $u^*\circ\Psi\circ u$, which we denote by $\res^{\scalebox{0.8}{$\gpd$}}_{\scalebox{0.8}{$\gpd_n$}}(\Psi)$.
  
    \begin{lem}\label{lem: rest_trace_is_multiplicative}
        Let $\trace_{\scalebox{0.8}{$\VN(\gpd)$}}$ and $\trace_{\scalebox{0.8}{$\VN(\gpd_n)$}}$ denote the corresponding trace function of $\gpd$ and its finite index subgroupoid $\gpd_n$, respectively. Then for every $\Psi\in \VN(\gpd)$ it holds
        \[
        \trace_{\scalebox{0.8}{$\VN(\gpd_n)$}}(\res^{\scalebox{0.8}{$\gpd$}}_{\scalebox{0.8}{$\gpd_n$}}(\Psi))=n\cdot \trace_{\scalebox{0.8}{$\VN(\gpd)$}}(\Psi).
        \]
    \end{lem}
    \begin{proof}
	    To begin with, this property needs only be checked on the dense subalgebra $\C[\gpd]$. So we assume that $\Psi\in\C[\gpd]$. We recall that 
        \[
        \trace_{\scalebox{0.8}{$\VN(\gpd_n)$}}(\res^{\scalebox{0.8}{$\gpd$}}_{\scalebox{0.8}{$\gpd_n$}}(\Psi))=\sum_{j=1}^n\trace_{\scalebox{0.8}{$\VN(\gpd_n)$}}(\res^{\scalebox{0.8}{$\gpd$}}_{\scalebox{0.8}{$\gpd_n$}}(\Psi)_{jj}).
        \]
        Let us compute $\trace_{\scalebox{0.8}{$\VN(\gpd_n)$}}(\res^{\scalebox{0.8}{$\gpd$}}_{\scalebox{0.8}{$\gpd_n$}}(\Psi)_{jj})$. First, $\res^{\scalebox{0.8}{$\gpd$}}_{\scalebox{0.8}{$\gpd_n$}}(\Psi)_{jj}(1)$ is given by taking the $j$th coordinate of $u^*\circ \Psi\circ u(0,\ldots,0, 1,0\ldots,0)$, where the identity function $1$ is in the $j$th position. Now $u$ maps $(0,\ldots, 0,1,0,\ldots 0)$ to $\chi_{A_j}$. So since $\Psi\in \VN(\gpd)$, by definition, we get that $\Psi(\chi_{A_j})=\chi_{A_j}\cdot\Psi(1)$. Thus $\res^{\scalebox{0.8}{$\gpd$}}_{\scalebox{0.8}{$\gpd_n$}}(\Psi)_{jj}(1)=(\chi_{A_j}\cdot\Psi(1))\circ F_j$. Hence, since $\gpd_n^j$ has full measure, $\res^{\scalebox{0.8}{$\gpd$}}_{\scalebox{0.8}{$\gpd_n$}}(\Psi)_{jj}(1)=\Psi(1)$ as $L^{\infty}(\gpd_n)$ functions by the definition of $F_j$. Therefore
        \[
        \trace_{\scalebox{0.8}{$\VN(\gpd_n)$}}(\res^{\scalebox{0.8}{$\gpd$}}(\Psi)_{jj})=\langle \res^{\scalebox{0.8}{$\gpd$}}(\Psi)_{jj}(1),1\rangle=\langle \Psi(1),1\rangle=\int_{\gpd^0}\Psi(x,1)d\mu(x,1),
        \]
        and so $\trace_{\scalebox{0.8}{$\VN(\gpd_n)$}}(\res^{\scalebox{0.8}{$\gpd$}}_{\scalebox{0.8}{$\gpd_n$}}(\Psi)_{ii})=\trace_{\scalebox{0.8}{$\VN(\gpd)$}}(\Psi)$, from which it follows the proof.
    \end{proof}

    We now show the multiplicative property of $L^2$-Betti numbers for these groupoids (see \cite[Proposition 5.11]{Gab02}).
        
    \begin{cor}\label{cor: multiplicity_betti_numb_gpds}
        For every non-negative integer $k$ it holds that $\beta_k^{(2)}(\gpd_n)=n\cdot \beta_k^{(2)}(\gpd)$.
    \end{cor}
    
    \begin{proof}
        The equality of trace functions from \cref{lem: rest_trace_is_multiplicative} translates into the fact that for a finitely presented left $\VN(\gpd)$-module $M$
        \[
        \dim_{\scalebox{0.8}{$\VN(\gpd)$}}(M)=\frac{\dim_{\scalebox{0.8}{$\VN(\gpd_n)$}}(M)}{n}.
        \]
        Thus since the extension to a BSMRF is uniquely determined by the values on finitely presented modules (cf. \cite[Lemma 3.5]{LiSMRF}), $\dim_{\scalebox{0.8}{$\VN(\gpd)$}}(\cdot)=\dim_{\scalebox{0.8}{$\VN(\gpd_n)$}}(\cdot)/n$. In particular, by exactness of the SMRF, we get $\beta_k^{(2)}(\gpd_n)=n\cdot \beta_k^{(2)}(\gpd)$ for every non-negative integer $k$.
    \end{proof}

    On the other side, given $\Psi\in \VN(\gpd_n)$ there is a natural way to associate an operator $\overline{\Psi}\in \VN(\gpd)$. Specifically, as in \cref{lem: gpd_ring_is_free_gpd_n_module} given $\phi\in L^2(\gpd)$ we can decompose $\phi$ as $\chi_{A_1}\cdot \phi_{1}+\ldots+\chi_{A_n}\cdot \phi_{n}$ where $\phi_j\in L^2(\gpd_n)$ with support in $\gpd_n^j$ is defined by $\phi\circ F_j$. Thus we set $\overline{\Psi}(\phi):=\chi_{A_1}\cdot \Psi(\phi_{1})+\ldots+\chi_{A_n}\cdot \Psi(\phi_{n})\in L^2(\gpd)$. We now show that in this case the trace is also preserved.

    \begin{lem}\label{lem: gpd_n_is_trace_preserving}
        With the above notation, $\trace_{\scalebox{0.8}{$\VN(\gpd_n)$}}(\Psi)=\trace_{\scalebox{0.8}{$\VN(\gpd)$}}(\overline{\Psi})$.
    \end{lem}
    \begin{proof}
        To begin with, this property needs only be checked on the dense subalgebra $\C[\gpd_n]$. So we assume that $\Psi\in\C[\gpd_n]$. We recall that $\trace_{\scalebox{0.8}{$\VN(\gpd)$}}(\overline{\Psi})=\langle \overline{\Psi}(1),1\rangle$.
        Observe that $1\circ F_j(\gamma)=1(x,g)$ where $\gamma=\gamma_{x,g}\in \gpd_n^j$. Hence, this is non-zero only if $g=1$, that is, if $(x,1)=\gamma\varphi_{j,j}(x,1)$ where $(x,1)\in Y_{j,j}^1$. So, if we set $\varphi_{j,j}(Y_{j,j}^1)^{-1}:=\{\varphi_{j,j}(x,1)^{-1}:(x,1)\in Y_{j,j}^1\}$, which is Borel by injectivity, then $1\circ F_j=\chi_{\varphi_{j,j}(Y_{j,j}^1)^{-1}}$. So by definition of $\VN(\gpd_n)$ 
        \begin{align*}
            \overline{\Psi}(1)&=\chi_{A_1}\cdot \Psi(\chi_{\varphi_{1,1}(Y_{1,1}^1)^{-1}})+\ldots+\chi_{A_n}\cdot \Psi(\chi_{\varphi_{n,n}(Y_{n,n}^1)^{-1}})\\
            &=\chi_{A_1}\cdot\chi_{\varphi_{1,1}(Y_{1,1}^1)^{-1}}\Psi(1)+\ldots+\chi_{A_n}\cdot \chi_{\varphi_{n,n}(Y_{n,n}^1)^{-1}}\Psi(1).
        \end{align*}
        Hence,
	    \begin{align*}
            \trace_{\scalebox{0.8}{$\VN(\gpd)$}}(\overline{\Psi})&=\sum_{j=1}^n \langle \chi_{A_j}\cdot\chi_{\varphi_{j,j}(Y_{j,j}^1)^{-1}}\Psi(1),1\rangle\\
		  &=\sum_{j=1}^n\langle \Psi(1), \chi_{\varphi_{j,j}(Y_{j,j}^1)}\cdot \chi_{A_j}^{*}\rangle.
	    \end{align*}
        We claim that $\chi_{\varphi_{j,j}(Y_{j,j}^1)}\cdot \chi_{A_j}^{*}=\chi_{Y_j}$. Indeed,
        \begin{align*}
            \chi_{\varphi_{j,j}(Y_{j,j}^1)}\cdot \chi_{A_j}^{*}(x,g)&=\sum_{(x_2,g_2)(x_1,g_1)=(x,g)}\chi_{\varphi_{j,j}(Y_{j,j}^1)}(x_1,g_1)\chi_{A_j}^*(x_2,g_2)\\
            &=\sum_{g_1\in G}\chi_{\varphi_{j,j}(Y_{j,j}^1)}(x,g_1)\chi_{A_j}((g_1\cdot x,gg_1^{-1})^{-1})\\
		  &=\sum_{g_1\in G}\chi_{\varphi_{j,j}(Y_{j,j}^1)}(x,g_1)\chi_{A_j}(g\cdot x,g_1g^{-1}).
	    \end{align*}
        Note that to obtain a non-zero value we need that both $(x,g_1)=\varphi_{j,j}(x,1)$ for some $(x,1)\in Y_{j,j}^1$ and $(g\cdot x, g_1g^{-1})=\varphi_{j,i}(g\cdot x,1)$ for some $(g\cdot x, 1)\in Y_{j,i}^1$. Now observe that $r(x,g_1)=r(g\cdot x, g_1g^{-1})$; thus $i=j$ and moreover, by the injectivity of $r\circ \varphi_{j,j}$, we get that $g\cdot x =x$. In particular,
        \[
        (x,g_1)=\varphi_{j,j}(x,1)=\varphi_{j,j}(g\cdot x,1)=(g\cdot x, g_1g^{-1}),
        \]
        that is, $g=1$. Hence, putting all together, $\chi_{\varphi_{j,j}(Y_{j,j}^1)}\cdot \chi_{A_j}^{*}= \chi_{Y_{j,j}^1}$, which agrees with $\chi_{Y_j}$ since $Y_{j,j}^1$ is a Borel subset of $Y_j$ with full measure. 
        
        Therefore
        \[
        \sum_{j=1}^n\langle \Psi(1), \chi_{\varphi_{j,j}(Y_{j,j}^1)}\cdot \chi_{A_j}^{*}\rangle=\sum_{j=1}^n\langle \Psi(1), \chi_{Y_j}\rangle=\langle \Psi(1),1\rangle
        \]
        which ends the proof.
    \end{proof}

    \begin{rem}
        Adapting the standard argument for groups to this setting, one shows that $\VN(\Hgpd)\hookrightarrow \VN(\gpd)$ is a trace-preserving $*$-homomorphism too.    
    \end{rem}
    
    It is well known that $L^2$-Betti numbers of $G$ and $\gpd$ coincide (see \cite[Section 4.5]{SauerThom}). However, the results of \cref{sec: hom_alg_dim_theory} enable us to reproduce the proof given by Sauer in \cite[Theorem 5.5]{SauerGroupoids} without restricting ourselves to groups acting essentially free on a probability space. When using the results from \cref{sec: hom_alg_dim_theory} there is always a list of facts that needs to be checked. The aim of the next lemma is to gather all the hypothesis that shall be verified in future applications of these dimension theorems. Here $N$ stands for the normal subgroup contained in $H$ from \cref{thm: groupoids_beta_k}.

    \begin{lem}\label{lem: dim_isomph_gpd_ring_crossed_ring}
        \begin{enumerate}[label=(\roman*)]
            \item The inclusions $L^{\infty}(\gpd^0)\ast G\subseteq \C[\gpd]$, $L^{\infty}(\gpd^0)\ast H\subseteq \C[\Hgpd]$ and $L^{\infty}(\gpd^0)\ast N\subseteq \C[X_{G,H}\rtimes N]$ are left $\dim_{\scalebox{0.8}{$\essbdd$}}$-isomorphisms.
            \item $\VN(\gpd)$ is $\dim_{\scalebox{0.8}{$\VN(\gpd)$}}$-$\dim_{\scalebox{0.8}{$\essbdd$}}$-compatible and $\VN(\gpd_n)$ is $\dim_{\scalebox{0.8}{$\VN(\gpd_n)$}}$-$\dim_{\scalebox{0.8}{$\essbdd$}}$-compatible.
            \item $\C[\gpd],\C[\gpd_n],\C[\Hgpd]$ and $\C[X_{G,H}\rtimes N]$ are $\dim_{\scalebox{0.8}{$\essbdd$}}$-$\dim_{\scalebox{0.8}{$\essbdd$}}$-compatible.
            \item $\VN(\gpd),\C[\gpd],\VN(\gpd_n),\C[\gpd_n],\VN(\Hgpd),\C[\Hgpd]$ and $\C[X_{G,H}\rtimes N]$ are all flat as right $\essbdd$-modules.
        \end{enumerate}
    \end{lem}
    \begin{proof}
        First, {\it (i)} follows by rewriting \cite[Lemma 5.4]{SauerGroupoids}. Second, {\it (ii)} is due to \cref{prop: dim_extension_for_finite_vNa}. Third, {\it (iii)} is shown as in \cite[Lemma 4.8]{SauerGroupoids}. Finally, {\it (iv)} follows from the fact that the corresponding von Neumann algebra is trace-preserving and hence flat by \cref{prop: trace_preserving_flat}. To conclude, since $\essbdd$ is a semihereditary ring according to \cref{thm: von_Neumann_alg_is_semihered}, submodules of flat modules are flat by \cref{prop: submod_flat_is_flat_if_semihered}.
    \end{proof}
		
    Henceforth, we shall use results from \cref{sec: hom_alg_dim_theory} without checking them explicitly; and thus we refer the reader to the above lemma in case of doubt.
  
    \begin{lem}\label{lem: equality_betti_numb_gps_gpds}
        For every non-negative integer $k$ it holds that
        \[
        \beta_k^{(2)}(\gpd)=\beta_k^{(2)}(G).
        \]
    \end{lem}
    \begin{proof}
        According to \cref{prop: VN_gpd_is_trace_preserving}, \cref{prop: dim_extension_for_finite_vNa} and \cref{prop: trace_preserving_flat}
        \[
        \beta_k^{(2)}(G)=\dim_{\scalebox{0.8}{$\VN(G)$}}\left(\Tor_k^{\scalebox{0.8}{$\C[G]$}}(\VN(G),\mathbb{C})\right)=\dim_{\scalebox{0.8}{$\VN(\gpd)$}}\left(\Tor_k^{\scalebox{0.8}{$\C[G]$}}(\VN(\gpd),\mathbb{C})\right).
        \]
        Besides, $L^{\infty}(\gpd^0)\ast G$ is a free right $\C[G]$-module; thus by Shapiro's lemma
        \[
        \dim_{\scalebox{0.8}{$\VN(\gpd)$}}\left(\Tor_k^{\scalebox{0.8}{$\C[G]$}}(\VN(\gpd),\mathbb{C})\right)
        \]
        equals to
        \[
        \dim_{\scalebox{0.8}{$\VN(\gpd)$}}\left(\Tor_k^{\scalebox{0.8}{$L^{\infty}(\gpd^0)\ast G$}}(\VN(\gpd),L^{\infty}(\gpd^0)\ast G\otimes_{\scalebox{0.8}{$\C[G]$}}\mathbb{C})\right),
        \]
        which is precisely
        \[
        \dim_{\scalebox{0.8}{$\VN(\gpd)$}}\left(\Tor_k^{\scalebox{0.8}{$L^{\infty}(\gpd^0)\ast G$}}(\VN(\gpd),L^{\infty}(\gpd^0))\right)
        \]
        according to the left $\C[\gpd]$-module structure of $L^{\infty}(\gpd^0)$. Finally we apply \cref{thm: dim_isomph_changing_scalars} for $L^{\infty}(\gpd^0)\ast G$ and $\C[\gpd]$; and hence
        \begin{align*}
            \dim_{\scalebox{0.8}{$\VN(\gpd)$}}\left(\Tor_k^{\scalebox{0.8}{$L^{\infty}(\gpd^0)\ast G$}}(\VN(\gpd),L^{\infty}(\gpd^0))\right)&=\dim_{\scalebox{0.8}{$\VN(\gpd)$}}\left(\Tor_k^{\scalebox{0.8}{$\C[\gpd]$}}(\VN(\gpd),L^{\infty}(\gpd^0))\right)\\
            &=\beta_k^{(2)}(\gpd).
        \end{align*}
	which ends the proof.
    \end{proof}

    \begin{rem}
        Note that the proof still works for a general translation groupoid $X\rtimes G$.
    \end{rem}
  
    To conclude this section we state a dimension equality with the flavour of Shapiro's lemma which will be of great importance.
  
    \begin{prop}\label{prop: shapiro_dim_for_gpds}
	For every non-negative integer $k$ it holds that
        \[
        \beta_k^{(2)}(H)=\dim_{\scalebox{0.8}{$\VN(\gpd_n)$}} \left(\Tor_k^{\scalebox{0.8}{$\C[\gpd_n]$}}(\VN(\gpd_n), \C[\gpd_n]\otimes_{\scalebox{0.8}{$\C[\Hgpd]$}}L^{\infty}(\gpd^0))\right).
        \]
    \end{prop}
    \begin{proof}
        First, recall that due to \cref{lem: equality_betti_numb_gps_gpds} $\beta_k^{(2)}(H)=\beta_k^{(2)}(\Hgpd)$. Besides, according to \cref{prop: trace_preserving_flat} $\VN(\gpd)$ is flat as right $\VN(\Hgpd)$-module. Hence by \cref{prop: dim_extension_for_finite_vNa}
        \[
        \beta_k^{(2)}(\Hgpd)=\dim_{\scalebox{0.8}{$\VN(\gpd)$}}\left(\Tor_n^{\scalebox{0.8}{$\C[\Hgpd]$}}(\VN(\gpd),L^{\infty}(\gpd^0))\right).
        \]
        We claim that by dimension flatness considerations
        \[
        \beta_k^{(2)}(\Hgpd)=\dim_{\scalebox{0.8}{$\VN(\gpd)$}}\left(\Tor_k^{\scalebox{0.8}{$\C[\gpd]$}}(\VN(\gpd),\C[\gpd]\otimes_{\scalebox{0.8}{$\C[\Hgpd]$}} L^{\infty}(\gpd^0))\right).
        \]
        Indeed, by \cref{thm: dim_isomph_changing_scalars}
        \[
        \dim_{\scalebox{0.8}{$\VN(\gpd)$}}\left(\Tor_k^{\scalebox{0.8}{$\C[\Hgpd]$}}(\VN(\gpd),L^{\infty}(\gpd^0))\right)=\dim_{\scalebox{0.8}{$\VN(\gpd)$}}\left(\Tor_k^{\scalebox{0.8}{$L^{\infty}(\gpd^0)\ast H$}}(\VN(\gpd),L^{\infty}(\gpd^0))\right).
        \]
        In addition, $L^{\infty}(\gpd^0)\ast G$ is free as right $\essbdd\ast H$-module. So by Shapiro's lemma
        \[
        \dim_{\scalebox{0.8}{$\VN(\gpd)$}}\left(\Tor_k^{\scalebox{0.8}{$L^{\infty}(\gpd^0)\ast H$}}(\VN(\gpd),L^{\infty}(\gpd^0))\right)
        \]
        equals to
        \[
        \dim_{\scalebox{0.8}{$\VN(\gpd)$}}\left(\Tor_k^{\scalebox{0.8}{$L^{\infty}(\gpd^0)\ast G$}}(\VN(\gpd),L^{\infty}(\gpd^0)\ast G\otimes_{\scalebox{0.8}{$L^{\infty}(\gpd^0)\ast H$}}L^{\infty}(\gpd^0))\right).
        \]
        On the one side, since $L^{\infty}(\gpd^0)\ast G\subseteq \C[\gpd]$ is a left $\dim_{\scalebox{0.8}{$L^{\infty}(\gpd^0)$}}$-isomorphism by \cref{lem: dim_isomph_gpd_ring_crossed_ring}, then according to \cref{lem: dim_isomph_tensor_right}
        \[
        L^{\infty}(\gpd^0)\ast G\otimes_{\scalebox{0.8}{$L^{\infty}(\gpd^0)\ast H$}}L^{\infty}(\gpd^0)\rightarrow \C[\gpd]\otimes_{\scalebox{0.8}{$L^{\infty}(\gpd^0)\ast H$}}L^{\infty}(\gpd^0)
        \]
        is a left $\dim_{\scalebox{0.8}{$L^{\infty}(\gpd^0)$}}$-isomorphism. On the other side, from \cref{thm: dim_isomph_changing_scalars} we get that
        \[
        \C[\gpd]\otimes_{\scalebox{0.8}{$L^{\infty}(\gpd^0)\ast H$}}L^{\infty}(\gpd^0)\rightarrow\C[\gpd]\otimes_{\scalebox{0.8}{$\C[\Hgpd]$}}L^{\infty}(\gpd^0)
        \]
        is a left $\dim_{\scalebox{0.8}{$L^{\infty}(\gpd^0)$}}$-isomorphism. Combining these left $\dim_{\scalebox{0.8}{$L^{\infty}(\gpd^0)$}}$-isomorphisms, by \cref{lem: comp_dim_isomph_is_dim_isomph}, we can apply \cref{lem: dim_isomph_tensor_left} to conclude that
        \[
        \dim_{\scalebox{0.8}{$\VN(\gpd)$}}\left(\Tor_k^{\scalebox{0.8}{$L^{\infty}(\gpd^0)\ast G$}}(\VN(\gpd),L^{\infty}(\gpd^0)\ast G\otimes_{\scalebox{0.8}{$L^{\infty}(\gpd^0)\ast H$}}L^{\infty}(\gpd^0))\right)
        \]
        equals to
        \[\dim_{\scalebox{0.8}{$\VN(\gpd)$}}\left(\Tor_k^{\scalebox{0.8}{$L^{\infty}(\gpd^0)\ast G$}}(\VN(\gpd),\C[\gpd]\otimes_{\scalebox{0.8}{$\C[\Hgpd]$}}L^{\infty}(\gpd^0))\right)
        \]
        since $L^{\infty}(\gpd^0)*G$ is $\dim_{\scalebox{0.8}{$L^{\infty}(\gpd^0)$}}$-$\dim_{\scalebox{0.8}{$L^{\infty}(\gpd^0)$}}$-compatible, for instance, by \cref{lem: dim_isomph_tensor_right} and \cref{lem: dim_isomph_gpd_ring_crossed_ring}. However, by \cref{thm: dim_isomph_changing_scalars}
        \[
        \dim_{\scalebox{0.8}{$\VN(\gpd)$}}\left(\Tor_k^{\scalebox{0.8}{$L^{\infty}(\gpd^0)\ast G$}}(\VN(\gpd),\C[\gpd]\otimes_{\scalebox{0.8}{$\C[\Hgpd]$}}L^{\infty}(\gpd^0))\right)
        \]
        and
        \[\dim_{\scalebox{0.8}{$\VN(\gpd)$}}\left(\Tor_k^{\scalebox{0.8}{$\C[\gpd]$}}(\VN(\gpd),\C[\gpd]\otimes_{\scalebox{0.8}{$\C[\Hgpd]$}}L^{\infty}(\gpd^0))\right)
        \]
        coincide, which shows the claim.
            
        Besides, according to \cref{lem: gpd_ring_is_free_gpd_n_module} $\C[\gpd]$ is free as right $\C[\gpd_n]$-module; and hence by Shapiro's lemma
        \[
        \beta_k^{(2)}(\Hgpd)=\dim_{\scalebox{0.8}{$\VN(\gpd)$}}\left(\Tor_k^{\scalebox{0.8}{$\C[\gpd_n]$}}(\VN(\gpd),\C[\gpd_n]\otimes_{\scalebox{0.8}{$\C[\Hgpd]$}}L^{\infty}(\gpd^0))\right).
        \]
        Finally, since $\VN(\gpd_n)\hookrightarrow\VN(\gpd)$ is a trace-preserving $\ast$-homomorphism by \cref{lem: gpd_n_is_trace_preserving}, then \cref{prop: trace_preserving_flat} and \cref{prop: dim_extension_for_finite_vNa} imply that
        \[
        \beta_k^{(2)}(\Hgpd)=\dim_{\scalebox{0.8}{$\VN(\gpd_n)$}}\left(\Tor_k^{\scalebox{0.8}{$\C[\gpd_n]$}}(\VN(\gpd_n),\C[\gpd_n]\otimes_{\scalebox{0.8}{$\C[\Hgpd]$}}L^{\infty}(\gpd^0))\right).
        \]
	\end{proof}

\section{Proof of Theorem \ref{thm: module_beta_k}}\label{sec: algebraic_proofs}

    To begin with, note that the vanishing criterion for $L^2$-Betti numbers of $G$ follows from \cref{thm: control_beta_k_subnormal_for_H} in case there are finite index subgroups of $G$ which have arbitrary high index and contain $N$. However, the existence of such families of subgroups are uncommon. To overcome this obstacle, Gaboriau, and Sauer and Thom use either equivalence relations or groupoids. It turns out that from an algebraic point of view, the module structure gives us the freedom to take advantage of the infinite index situation. The key idea is to consider the left $\UO(N)$-module $\UO(N)\otimes_{\scalebox{0.8}{$\C[N]$}}\C[G]$. This object actually has a ring structure given by the crossed product $\UO(N)*G/N$ where the group action comes from conjugation. More in detail, since every $g\in G$ normalizes $\C[N]$, then every $g$ normalizes $\VN(N)$ seen as the strong closure in the bounded operators of $\ell^2(N)\subseteq \ell^2(G)$. Hence, every $g$ also normalizes its Ore localization $\UO(N)$ inside $\UO(G)$. In addition, elements in different cosets of $N$ in $G$ are right (and thus left since $g\UO(N)g^{-1}=\UO(N)$ for every $g\in G$) $\UO(N)$-linearly independent. Therefore, it makes sense to consider the crossed product $\UO(N) * G/N$ that can be realized inside $\UO(G)$ (see \cite[Lemma 10.57]{Luck02}). To sum up, one has
    \[
    \C[G]=\C[N]*G/N\subseteq \UO(N)*G/N \subseteq \UO(G).
    \]
    Our algebraic proof builds heavily on the next dimension criterion.
    
    \begin{thm}\label{thm: finite_dim_ind_zero_dim}
        Let $G$ be a countable group, $N$ an infinite index normal subgroup in $G$ and denote by $S$ the ring $\UO(N)* G/N$. Suppose that $M$ is a left $S$-module with $\dim_{\scalebox{0.8}{$\UO(N)$}}(M)$ finite. Then
        \[
        \dim_{\scalebox{0.8}{$\UO(G)$}}(\UO(G)\otimes_S M)=0.
        \]
    \end{thm}

    \begin{proof}
        First, we assume that $M$ is a finitely generated left $S$-module; so $M\cong S^k/L$ as $S$-module for some non-negative integer $k$ and some $S$-submodule $L$ of $S^k$. Without loss of generality we can assume that $M$ is cyclic, that is, $k=1$. Indeed, there exist $s_1,\ldots, s_k\in S$ such that $s_1+L,\ldots, s_k+L$ generate $M$. Observe that there is a natural surjective left $S$-homomorphism
        \[
        Ss_1+L/L \oplus \ldots \oplus Ss_k+L/L \rightarrow M
        \]
        where $Ss_i+L/L$ is a $S$-submodule of $M$. But, $Ss_i+L/L$ is also an $\UO(N)$-submodule of $M$, and hence, since $\dim_{\scalebox{0.8}{$\UO(N)$}}$ is exact, $Ss_i+L/L$ has finite dimension as $\UO(N)$-module. Finally, since extending scalars is right exact, it suffices to prove the claim for cyclic modules by \cref{defn: BSMRF}.
            
        Now let $T$ be a left transversal of $N$ in $G$ and consider the following family $\{T_i\}_{i\geq 0}$ of non-empty finite subsets of $T$ ordered by inclusion and such that $T=\cup_{i\geq 0}T_i$. Fix some $T_i=\{t_1,\ldots, t_{n_i}\}$. We shall show that
        \[
        \dim_{\scalebox{0.8}{$\UO(G)$}}(\UO(G)\otimes_S M)\leq \frac{\dim_{\scalebox{0.8}{$\UO(N)$}}(M)}{|T_i|}+\frac{1}{|T_i|^2}.
        \]
        Denote by $\UO(N)[T_i]$ the $\UO(N)$-module defined as $\oplus_{j=1}^{n_i}\UO(N)\overline{t_j}$ where $\ \bar{}\ $ stands for the crossed product construction with $G/N$. Then
        \begin{align*}
        	\dim_{\scalebox{0.8}{$\UO(N)$}}(M)&\geq \dim_{\scalebox{0.8}{$\UO(N)$}}\left(\UO(N)[T_i]/(L\cap\UO(N)[T_i])\right)\\
        	&=|T_i|-\dim_{\scalebox{0.8}{$\UO(N)$}}\left(L\cap \UO(N)[T_i]\right)
        \end{align*}
        since $\UO(N)[T_i]/(L\cap\UO(N)[T_i])$ is an $\UO(N)$-subquotient of $M$. Besides, the dimension function is exact; and so
        \[
        \dim_{\scalebox{0.8}{$\UO(N)$}}\left(L\cap \UO(N)[T_i]\right)=\sup_K\dim_{\scalebox{0.8}{$\UO(N)$}}(K)
        \]
        for $K$ ranging over the finitely generated $\UO(N)$-submodules of $L\cap\UO(N)[T_i]$. Thus given $\varepsilon_i=1/|T_i|$, there are $a_1,\ldots,a_l\in L\cap\UO(N)[T_i]$ such that $K_i=\UO(N) a_1+\ldots+\UO(N) a_l$ satisfies 
        \[
        \dim_{\scalebox{0.8}{$\UO(N)$}}\left(L\cap \UO(N)[T_i]\right)\leq \dim_{\scalebox{0.8}{$\UO(N)$}}(K_i)+\frac{1}{|T_i|},
        \]
	    and so
        \[
        \dim_{\scalebox{0.8}{$\UO(N)$}}(M)\geq |T_i|-\dim_{\scalebox{0.8}{$\UO(N)$}}(K_i)-\frac{1}{|T_i|}.
        \]
    
        Moreover, since $\mathcal{U}(N)$ is the Ore localization of $\mathcal{N}(N)$, we can assume without loss of generality that $a_1,\ldots, a_l\in \VN(N)[T_i]$ according to \cref{lem: common_denominator_ore}. In particular, by \cref{lem: UOdim_equal_vNdim_for_elements_in_vN}
        \[
        \dim_{\scalebox{0.8}{$\UO(N)$}} (K_i)=\dim_{\scalebox{0.8}{$\mathcal{N}(N)$}}(\VN(N) a_1+\ldots+\VN(N) a_l)
        \]
        \begin{claim}\label{claim: from_dimN_to_dimG}
            We have that $\dim_{\scalebox{0.8}{$\UO(N)$}} (K_i)\leq |T_i|\dim_{\scalebox{0.8}{$\UO(G)$}}(\UO(G) a_1+\ldots+\UO(G) a_l)$.
        \end{claim}
        \begin{proof}
        To this end we shall define two Hilbert spaces. Firstly, note that $\mathcal{N}(N)[T_i]$ embeds diagonally into $\ell^2(N)^{|T_i|}$ as a Hilbert $N$-module. Thus if we denote by $\ell^2(N)a_j$ the embedded image of $\VN(N)a_j$, our first Hilbert subspace $W$ is defined as follows
        \[
        W:=\overline{\ell^2(N)a_1+\ldots +\ell^2(N)a_l}\subseteq \ell^2(N)^{|T_i|}.
        \]
	    In the same fashion, $V$ is the second Hilbert subspace defined as
        \[
        V:=\overline{\ell^2(G)a_1+\ldots +\ell^2(G)a_l}\subseteq \ell^2(G).
        \]
	    Therefore we have
 
        \begin{align*}
		  \dim_{\scalebox{0.8}{$\UO(N)$}}(K_i)&=\dim_{\scalebox{0.8}{$\mathcal{N}(N)$}}(W)=\sum_{j=1}^{|T_i|}\langle \pr_W(1_j), 1_j\rangle\\
		  \dim_{\scalebox{0.8}{$\UO(G)$}}(\mathcal{U}(G) a_1+\ldots +\mathcal{U}(G) a_l)& =\dim_{\scalebox{0.8}{$\mathcal{N}(G)$}}(V)=\langle \pr_V(1),1\rangle
	    \end{align*}
 
        where pr$_W$ and pr$_V$ are the orthogonal projections onto the subspaces $W$ and $V$, respectively. Note that $\ell^2(N)^{|T_i|}$ embeds isometrically into $\ell^2(G)$ multiplying each coordinate by the appropriate $t\in T_i$; and hence so does $W$ into $V$. Thus it holds
 
	    \begin{align*}
		  \sum_{j=1}^{|T_i|}\langle \pr_W(1_j), 1_j\rangle& =\sum_{t\in T_i}\langle \pr_{i(W)}(t),t\rangle \leq \sum_{t\in T_i}\langle \pr_{V}(t),t\rangle\\
		  &=|T_i|\langle \pr_V(1),1 \rangle
	    \end{align*}
 
        where the inequality follows from the fact that $i(W)$ is a subspace of $V$, and the last equality holds because multiplying by an element of $G$ is an isometry.
        \end{proof}
        
        Then by \cref{claim: from_dimN_to_dimG}
        \[
        \dim_{\scalebox{0.8}{$\UO(N)$}}(M)\geq |T_i|-|T_i|\dim_{\scalebox{0.8}{$\UO(G)$}}(\UO(G) a_1+\ldots +\UO(G) a_l)-\frac{1}{|T_i|},
        \]
	    that is,
	    \begin{equation*}
		  \frac{\dim_{\scalebox{0.8}{$\UO(N)$}}(M)}{|T_i|}+\frac{1}{|T_i|^2}\geq 1-\dim_{\scalebox{0.8}{$\UO(G)$}}(\UO(G) a_1+\ldots +\UO(G) a_l).
	    \end{equation*}
        \begin{claim}\label{claim: from_dimS_to_dimG}
            We have that $\dim_{\scalebox{0.8}{$\UO(G)$}}(\UO(G)\otimes_S M)\leq 1-\dim_{\scalebox{0.8}{$\UO(G)$}}(\UO(G)a_1+\ldots+\UO(G)a_l)$.
        \end{claim}
        \begin{proof}
	       To begin with 
	    \begin{align*}
		  \dim_{\scalebox{0.8}{$\UO(G)$}}(\UO(G)\otimes_S M)&=\dim_{\scalebox{0.8}{$\UO(G)$}}\left(\UO(G)\otimes_S (S/L)\right)\\
		  &=\dim_{\scalebox{0.8}{$\UO(G)$}}\left((\UO(G)\otimes S)/\UO(G) L\right)
	    \end{align*}
        where $\UO(G) L$ denotes the image submodule. Now with abuse of notion on $\UO(G) L$ it holds that
	    \begin{align*}
		  \dim_{\scalebox{0.8}{$\UO(G)$}}\left((\UO(G)\otimes S)/\UO(G) L\right)&=\dim_{\scalebox{0.8}{$\UO(G)$}}\left(\UO(G)/\UO(G) L\right)\\
		  &=1-\dim_{\scalebox{0.8}{$\UO(G)$}}(\UO(G) L).
	    \end{align*}
        Thus combining the above two equalities we get that
        \[
        \dim_{\scalebox{0.8}{$\UO(G)$}}(\UO(G)\otimes_S M)=1-\dim_{\scalebox{0.8}{$\UO(G)$}}(\UO(G) L).
        \]
	    To conclude the claim observe that
        \[
        \dim_{\scalebox{0.8}{$\UO(G)$}}(\UO(G) a_1+\ldots +\UO(G) a_l)\leq \dim_{\scalebox{0.8}{$\UO(G)$}}(\UO(G) L)
        \]
	    since $\dim_{\scalebox{0.8}{$\UO(G)$}}$ is exact. 
        \end{proof}
        Hence, by \cref{claim: from_dimS_to_dimG}
        \[
        \dim_{\scalebox{0.8}{$\UO(G)$}}(\UO(G)\otimes_S M)\leq \frac{\dim_{\scalebox{0.8}{$\UO(N)$}}(M)}{|T_i|}+\frac{1}{|T_i|^2}
        \]
        as we wanted. Finally, since $T_i$ was arbitrary and $\dim_{\scalebox{0.8}{$\UO(N)$}}(M)$ is finite
        \[
        \dim_{\scalebox{0.8}{$\UO(G)$}}(\UO(G)\otimes_S M)\leq\inf_{i\geq 0}\left\{\frac{\dim_{\scalebox{0.8}{$\UO(N)$}}(M)}{|T_i|}+\frac{1}{|T_i|^2}\right\}=0
        \]
	    which ends the proof of the finitely generated case.
 
        Now we show the standard argument to reduce the general case to the finitely generated one. Let $M$ be a left $S$-module. According to \cref{lem: dim_directed_union}
        \[
        \dim_{\scalebox{0.8}{$\UO(G)$}}(\UO(G)\otimes_S M)\leq \sup_{M'}\dim_{\scalebox{0.8}{$\UO(G)$}}(\UO(G) \otimes_S M')
        \]
        where $M'$ ranges over finitely generated $S$-submodules of $M$. Choose an arbitrary such $M'$. In particular, it is a $\UO(N)$-submodule of $M$; and hence $\dim_{\scalebox{0.8}{$\UO(N)$}}(M')$ is finite by exactness. Consequently, by the finitely generated case $\dim_{\scalebox{0.8}{$\UO(G)$}}(\UO(G)\otimes_S M')=0$; and since $M'$ was arbitrary, we conclude that
        \[
        \dim_{\scalebox{0.8}{$\UO(G)$}}(\UO(G)\otimes_S M)=0
        \]
	    which completes the proof.
    \end{proof}

    In order to show \cref{thm: module_beta_k} we must have some control on the dimension of certain modules arising from the long exact sequences given by Tor. The following result states that in terms of dimension the modules appearing on these long exact sequences behave as zero modules.
        
    \begin{prop}\label{prop: control_dim_modules}
        Let $G$ be a countable group, $N$ a normal subgroup in $G$, $M$ a left $\C[G]$-module in which $N$ acts trivially and $k$ be a non-negative integer. Suppose that $\beta_i^{(2)}(N)$ vanishes for all $0\leq i\leq k$. Then
        \[
        \dim_{\scalebox{0.8}{$\UO(G)$}}\left(\Tor_i^{\scalebox{0.8}{$\C[G]$}}(\UO(G),M)\right)=0
        \]
	for all $0\leq i\leq k$.
    \end{prop}
        
    \begin{proof}
        Let $M$ be a left $\C[G]$-module in which $N$ acts trivially. Then we have a short exact sequence of $\C[G]$-modules
        \[
        0\longrightarrow L\longrightarrow (\C[G]\otimes_{\scalebox{0.8}{$\mathbb{C}[N]$}} \mathbb{C})^{I}\longrightarrow M\longrightarrow 0
        \]
        for some cardinal number $I$. We argue by induction on $k$. For $k=0$, extending scalars on the above sequence yields to a surjective map
        \[
        \UO(G)\otimes_{\scalebox{0.8}{$\mathbb{C}[G]$}}(\C[G]\otimes_{\scalebox{0.8}{$\mathbb{C}[N]$}}\mathbb{C})^I\rightarrow \UO(G)\otimes_{\scalebox{0.8}{$\mathbb{C}[G]$}} M.
        \]
        Now observe that
        \[
        \dim_{\scalebox{0.8}{$\UO(G)$}}\left(\UO(G)\otimes_{\scalebox{0.8}{$\mathbb{C}[G]$}}(\C[G]\otimes_{\scalebox{0.8}{$\mathbb{C}[N]$}}\mathbb{C})^I\right)
        \]
        equals to
        \[
        \dim_{\scalebox{0.8}{$\UO(G)$}}\left(\oplus_I\UO(G)\otimes_{\scalebox{0.8}{$\mathbb{C}[G]$}}(\C[G]\otimes_{\scalebox{0.8}{$\mathbb{C}[N]$}}\mathbb{C})\right).
        \]
        However, by exactness of dimension and \cref{lem: beta_H_as_dimG}
        \[
        \dim_{\scalebox{0.8}{$\UO(G)$}}\left(\oplus_I\UO(G)\otimes_{\scalebox{0.8}{$\mathbb{C}[G]$}}\C[G]\otimes_{\scalebox{0.8}{$\mathbb{C}[N]$}}\mathbb{C}\right)
        \]
        is at most
        \[
        \sup_{n\in \N}n\cdot \dim_{\scalebox{0.8}{$\UO(G)$}}\left(\UO(G)\otimes_{\scalebox{0.8}{$\mathbb{C}[G]$}}\C[G]\otimes_{\scalebox{0.8}{$\mathbb{C}[N]$}}\mathbb{C}\right)=\sup_{n\in \N}n\cdot \beta_0^{(2)}(N)=0.
        \]
        Therefore, according to \cref{defn: BSMRF} $\dim_{\scalebox{0.8}{$\UO(G)$}}(\UO(G)\otimes_{\scalebox{0.8}{$\mathbb{C}[G]$}} M)=0$ which shows the base step.

        Now assume the statement holds for $k-1$. Consider the long exact sequence associated to the initial exact sequence
        \begin{align*}
            &\mbox{Tor}_{k}^{\scalebox{0.8}{$\C[G]$}}(\UO(G),(\C[G]\otimes_{\scalebox{0.8}{$\mathbb{C}[N]$}}\mathbb{C})^I)\rightarrow\mbox{Tor}_{k}^{\scalebox{0.8}{$\C[G]$}}(\UO(G),M)\rightarrow\ldots\\
            &\ldots\rightarrow \UO(G)\otimes_{\scalebox{0.8}{$\mathbb{C}[G]$}}L\rightarrow\UO(G)\otimes_{\scalebox{0.8}{$\mathbb{C}[G]$}}(\C[G]\otimes_{\scalebox{0.8}{$\mathbb{C}[N]$}}\mathbb{C})^I\rightarrow \UO(G)\otimes_{\scalebox{0.8}{$\mathbb{C}[G]$}} M\rightarrow 0.
        \end{align*}
        Note that the dimension of all the modules from $\mbox{Tor}_{k-1}^{\scalebox{0.8}{$\C[G]$}}(\UO(G),L)$ to $\UO(G)\otimes_{\scalebox{0.8}{$\mathbb{C}[G]$}}M$, including both of them, is zero by induction hypothesis since $N$ is a normal subgroup. Hence taking dimensions in the long exact sequence we get
        \[
        \dim_{\scalebox{0.8}{$\UO(G)$}}\left(\mbox{Tor}_{k}^{\scalebox{0.8}{$\C[G]$}}(\UO(G),M)\right)\leq \dim_{\scalebox{0.8}{$\UO(G)$}}\left(\mbox{Tor}_{k}^{\scalebox{0.8}{$\C[G]$}}(\UO(G),(\C[G]\otimes_{\scalebox{0.8}{$\mathbb{C}[N]$}}\mathbb{C})^I)\right).
        \]
        Arguing as above we conclude that the right hand side vanishes, and hence it holds for $k$ which concludes the induction.
    \end{proof}

    An interesting consequence of this result is the fact that produces a Sylvester module rank function for $\C[G/N]$:
        
    \begin{cor}
        Let $G$ be a countable group, $N$ a normal subgroup in $G$ and $k$ a positive integer. Suppose that $\beta_i^{(2)}(N)$ vanishes for all $0\leq i\leq k-1$ and $\beta_k^{(2)}(N)$ is finite and non-zero. Then we can define a Sylvester module rank function by assigning to any finitely presented $\C[G/N]$-module $M$ the value
        \[
        \dim_{G,N}(M):=\frac{\dim_{\scalebox{0.8}{$\UO(G)$}}\left(\Tor_k^{\scalebox{0.8}{$\C[G]$}}(\UO(G),M)\right)}{\beta_k^{(2)}(N)}.
        \]
    \end{cor}
        
    \begin{proof}
        All conditions from \cref{defn: SMRF} follow easily except for the last one. So assume that 
        \[
        M_1\xrightarrow{f} M_2\xrightarrow{g} M_3\rightarrow 0
        \]
        is an exact sequence of $\mathbb{C}[G/N]$-modules. Then we have two associated short exact sequences
        \[
        0\rightarrow \ker g\rightarrow M_2\rightarrow M_3\rightarrow 0\quad\mbox{and}\quad 0\rightarrow \ker f\rightarrow M_1\rightarrow \im f\rightarrow 0.
        \]
        Since $\ker g = \im f$, taking dimensions in the long exact sequences derived by extending scalars along with \cref{prop: control_dim_modules} ends the proof.
    \end{proof}

    \begin{rem}
        This last result gives a new interpretation of Sauer and Thom's theorem. Indeed, if $\beta_k^{(2)}(N)$ vanishes, then taking $M=\mathbb{C}$ in \cref{prop: control_dim_modules} implies that $\beta_k^{(2)}(G)$ vanishes too. Otherwise, we are reduce to show that the above Sylvester module rank function $\dim_{G,N}$ vanishes at $\mathbb{C}$.
    \end{rem}

    We are now ready to present a proof of \cref{thm: module_beta_k} which relies in a combination of \cref{prop: control_dim_modules}, which ensures that certain module has finite dimension, and \cref{thm: finite_dim_ind_zero_dim}, to compress the dimension to zero.
        
    \begin{proof}[Proof of \cref{thm: module_beta_k}]
        Let $M$ be an arbitrary left $\mathbb{C}[G]$-module in which $N$ acts trivially. Then we have a short exact sequence of $\C[G]$-modules
        \[
        0\rightarrow J\rightarrow (\C[G]\otimes_{\scalebox{0.8}{$\mathbb{C}[N]$}} \mathbb{C})^{I}\rightarrow M\longrightarrow 0
        \]
        for some cardinal number $I$. Let $S$ be the subring $\UO(N)*G/N$ of $\UO(G)$. We apply first $S\otimes_{\scalebox{0.8}{$\C[G]$}}$ and second $\UO(G)\otimes_{\scalebox{0.8}{$\C[G]$}}$ to get long exact sequences
        \begin{equation}
            \ldots \rightarrow\Tor_k^{\scalebox{0.8}{$\C[G]$}}(S,(\C[G]\otimes_{\scalebox{0.8}{$\mathbb{C}[N]$}} \mathbb{C})^{I})\xrightarrow{i_k}\Tor_k^{\scalebox{0.8}{$\C[G]$}}(S,M)\rightarrow\ldots\label{LESS}
        \end{equation}
        and
        \begin{equation}
            \ldots \rightarrow\Tor_k^{\scalebox{0.8}{$\C[G]$}}(\UO(G),(\C[G]\otimes_{\scalebox{0.8}{$\mathbb{C}[N]$}} \mathbb{C})^{I})\xrightarrow{j_k}\Tor_k^{\scalebox{0.8}{$\C[G]$}}(\UO(G),M)\rightarrow\ldots\label{LESUG}
        \end{equation}
        Denote by $J$ and $K$ the image of $i_k$ and $j_k$, respectively.
        \begin{claim}\label{claim: dimK_beta_k_M}
            We have that $\dim_{\scalebox{0.8}{$\UO(G)$}}(K)=\dim_{\scalebox{0.8}{$\UO(G)$}}\left(\Tor_k^{\scalebox{0.8}{$\C[G]$}}(\UO(G),M)\right)$.
        \end{claim}
        \begin{proof}
            Note that $\beta_i^{(2)}(N)=0$ for all $0\leq i\leq k-1$. So according to \cref{prop: control_dim_modules} the dimension of all the modules at the right hand side of (\ref{LESUG}) is zero; and thus by exactness of dimension we conclude
            \[
            \dim_{\scalebox{0.8}{$\UO(G)$}}\left(\Tor_k^{\scalebox{0.8}{$\C[G]$}}(\UO(G),M)\right)=\dim_{\scalebox{0.8}{$\UO(G)$}}(K).
            \]
        \end{proof}
        \begin{claim}\label{claim: dimL_is_finite}
            We have that $\dim_{\scalebox{0.8}{$\UO(N)$}}(J)$ is finite.
        \end{claim}
        \begin{proof}
            Since $J$ is a $S$-submodule of $\Tor_k^{\scalebox{0.8}{$\C[G]$}}(S,M)$, in particular, it is an $\UO(N)$-submodule. Then by Shapiro's lemma along with exactness we get
            \[
            \dim_{\scalebox{0.8}{$\UO(N)$}}(J)\leq \dim_{\scalebox{0.8}{$\UO(N)$}}\left(\Tor_k^{\scalebox{0.8}{$\C[G]$}}(S,M)\right)=\dim_{\scalebox{0.8}{$\UO(N)$}}\left(\Tor_k^{\scalebox{0.8}{$\C[N]$}}(\UO(N),M)\right)
            \]
            which is finite by hypothesis.
        \end{proof}
        \begin{claim}\label{claim: dim_S_L_controls_dim_K}
            We have that $\dim_{\scalebox{0.8}{$\UO(G)$}}(\UO(G)\otimes_S J)\geq \dim_{\scalebox{0.8}{$\UO(G)$}}(K)$.
        \end{claim}
        \begin{proof}
            First observe that by (\ref{lem: commutative_diagram_Tor}) we have the commutative diagram
            \[
            \begin{tikzcd}
                \UO(G) \otimes_S\Tor_k^{\scalebox{0.8}{$\C[G]$}}(S,(\C[G]\otimes_{\scalebox{0.8}{$\mathbb{C}[N]$}} \mathbb{C})^{I})\arrow{r}\arrow{d}&\UO(G)\otimes_S J\arrow{d}\\
                \Tor_k^{\scalebox{0.8}{$\C[G]$}}(\UO(G),(\C[G]\otimes_{\scalebox{0.8}{$\mathbb{C}[N]$}} \mathbb{C})^{I})\arrow{r}&K
            \end{tikzcd}.
            \]
            Recall that we know that both horizontal maps are surjective. Therefore, a chase in the diagram shows that the right vertical map is surjective provided that the left vertical map is an isomorphism, and that is the case. Indeed, by Shapiro's lemma
            \[
            \UO(G)\otimes_S\Tor_k^{\scalebox{0.8}{$\C[G]$}}(S,(\C[G]\otimes_{\scalebox{0.8}{$\mathbb{C}[N]$}} \mathbb{C})^{I})\cong \UO(G)\otimes_S \Tor_k^{\scalebox{0.8}{$\C[N]$}}(S,\mathbb{C}^I).
            \]
            Moreover, $\UO(N)\subseteq S$ is a regular subring, so
            \begin{align*}
                \UO(G)\otimes_S \Tor_k^{\scalebox{0.8}{$\C[N]$}}(S,\mathbb{C}^I)&\cong \UO(G)\otimes_S \Tor_k^{\scalebox{0.8}{$\C[N]$}}(S\otimes_{\scalebox{0.8}{$\UO(N)$}}\UO(N),\mathbb{C}^I)\\
                &\cong \UO(G)\otimes_S S\otimes_{\scalebox{0.8}{$\UO(N)$}} \Tor_k^{\scalebox{0.8}{$\C[N]$}}(\UO(N),\mathbb{C}^I)\\
                &\cong \UO(G)\otimes_{\scalebox{0.8}{$\UO(N)$}} \Tor_k^{\scalebox{0.8}{$\C[N]$}}(\UO(N),\mathbb{C}^I)\\
                &\cong \Tor_k^{\scalebox{0.8}{$\C[N]$}}(\UO(G)\otimes_{\scalebox{0.8}{$\UO(N)$}} \UO(N),\mathbb{C}^I)\\
                &\cong \Tor_k^{\scalebox{0.8}{$\C[N]$}}(\UO(G),\mathbb{C}^I).
            \end{align*}
            Finally, once again by Shapiro's lemma
            \[
            \Tor_k^{\scalebox{0.8}{$\C[N]$}}(\UO(G),\mathbb{C}^I)\cong \Tor_k^{\scalebox{0.8}{$\C[G]$}}(\UO(G),(\C[G]\otimes_{\scalebox{0.8}{$\mathbb{C}[N]$}} \mathbb{C})^{I}),
            \]
            which shows the claim.
        \end{proof}
        Now by \cref{claim: dimL_is_finite} and \cref{thm: finite_dim_ind_zero_dim}, $\dim_{\scalebox{0.8}{$\UO(G)$}}(\UO(G)\otimes_S J)=0$. However, according to \cref{claim: dim_S_L_controls_dim_K} $\dim_{\scalebox{0.8}{$\UO(G)$}}(\UO(G)\otimes_S J)\geq \dim_{\scalebox{0.8}{$\UO(G)$}}(K)$. Therefore, $\dim_{\scalebox{0.8}{$\UO(G)$}}(K)$ vanishes, and so it does $\dim_{\scalebox{0.8}{$\UO(G)$}}\Tor_k^{\scalebox{0.8}{$\C[G]$}}(\UO(G),M)$ by \cref{claim: dimK_beta_k_M}.
    \end{proof}

\section{New Vanishing Criteria}\label{sec: new_vanishing_crit}
    
    Firstly, we show \cref{thm: groupoids_beta_k}. The idea of the proof is a combination of the Peterson and Thom's approach and a new algebraic point of view. Recall that $\gpd$ stands for the translation groupoid $X_{G,H}\rtimes G$, $\Hgpd$ for the translation subgroupoid $X_{G,H}\rtimes H$ and $\gpd_n$ for the intermediate subgroupoids of \cref{subsect: alg_struct_gpds} where $H$ is the infinite index subgroup of $G$ from \cref{thm: groupoids_beta_k} and $X_{G,H}=\prod_{gH\in G/H}[0,1]$ with $G$ acting by left multiplication. In contrast to the group situation, we now have at our disposal a family of finite index subgroupoids which have arbitrary high index and contain $\Hgpd$. Specifically, since the index of $H$ in $G$ is infinite, for each positive integer $n$ there is a subgroupoid $\gpd_n$ of $\gpd$ of index $n$ containing $\Hgpd$ (see \cref{subsect: alg_struct_gpds}). Moreover, by \cref{cor: multiplicity_betti_numb_gpds}, for each positive integer $k$ it holds that
    \[
    \beta_k^{(2)}(\gpd)=\frac{\beta_k^{(2)}(\gpd_n)}{n}.
    \]
    In addition, according to \cref{lem: equality_betti_numb_gps_gpds}, $\beta_k^{(2)}(H)=\beta_k^{(2)}(\Hgpd)$ and $\beta_k^{(2)}(G)=\beta_k^{(2)}(\gpd)$. Thus the proof would conclude if we can show that under the theorem's conditions
    \begin{equation}\label{eqtn: mult_betti_Hgpd_and_gpd_n}
        \beta_k^{(2)}(\Hgpd)\geq \beta_k^{(2)}(\gpd_n).
    \end{equation}
    This is exactly the point of \cite[Theorem 6.9]{PetersonThom_Freiheit} for $k=1$ based on $\gpd$-cohomology. Instead of using $\gpd$-cocycles, we suggest a module theoretic approach. Before proving the above inequality we have to establish two preliminary results. Recall that $N$ is a normal subgroup in $G$ that is contained in $H$.

    \begin{lem}\label{lem: trivial_action}
        With the above notation, $N$ acts trivially on the left $\C[\gpd_n]$-module $\C[\gpd_n]\otimes_{\scalebox{0.8}{$\mathbb{C}[X_{G,H}\rtimes N]$}}L^{\infty}(\gpd^0)$.
    \end{lem}

    \begin{proof}
        As in \cref{lem: gpd_ring_are_finite_sums}, every $\phi\in \C[\gpd_n]$ can be written as a finite sum $\phi=\sum_{j=1}^m f_j\cdot \chi_{E_j}$, where $f_j\in \essbdd$ and each Borel subset $E_j\subseteq \gpd_n$ has the property that $s\vrule_{E_j}$ and $r\vrule_{E_j}$ are injective. Note that since $N$ fixes all points of $X_{G,H}$, then $l\cdot f_j=f_j\cdot l$ for all $l\in N$ and all $1\leq j \leq m$. So it suffices to show that any given $l\in N$ acts trivially on elements of the form $\chi_E\otimes f$ with $E\subseteq \gpd_n$ a Borel subset such that $s\vrule_{E}$ and $r\vrule_{E}$ are injective, and $f\in \essbdd$. 
                
        First, observe that since $s$ is injective on $E$, then $E$ has the following form $\{(x,g_x):x\in A, g_x\in G\}$ for some $A\subseteq X_{G,H}$. Hence, $l\cdot \chi_E(y,h)=\chi_E(y,hl)=\chi_{El^{-1}}(y,h)$ where $El^{-1}:=\{(x,g_xl^{-1}):(x,g_x)\in E\}$, which is Borel. Now, since $N$ is a normal subgroup of $G$, for each $g_x$ there is some $l_x\in N$ such that $g_xl^{-1}=l_x g_x$. Thus we set $E_l:=\{(g_x x,l_x):(x,g_x)\in E\}\subseteq \gpd_n$. Note that $E_l$ is a Borel subset. Indeed, since the composition map in $\gpd_n$, call it $\theta$, is Borel, we have that
        \[
        \theta^{-1}(El^{-1})=\{((h\cdot x, l_x g_x h^{-1}),(x,h)): x\in A, (x,h)\in \gpd_n\}
        \] 
        is Borel. Moreover, $\gpd_n\times E$ is also a Borel subset of $\gpd_n\times \gpd_n$; and thus so is
        \[
        \theta^{-1}(El^{-1})\cap \gpd_n\times E=\{((g_x \cdot x,l_x),(x,g_x)):x\in A\},
        \]      
        where the equality follows from the fact that $s(x,h)=s(x,g_x)$ and the source map is injective on $E$. Finally, since projecting onto the first coordinate is a Borel map, and this projection map is injective on $\theta^{-1}(El^{-1})\cap \gpd_n\times E$ because $r$ is injective on $E$, the image of the projection map onto the first coordinate, which is precisely $E_l$, is also Borel \cite[Corollary 15.2]{Kech_setth}.

        Therefore, $\chi_{E_l}\in \mathbb{C}[X_{G,H}\rtimes N]$ and satisfies $\chi_{El^{-1}}=\chi_E\cdot \chi_{E_l}$. Moreover, since $s$ is injective on $E_l$, because $r$ is injective on $E$, again by \cite[Corollary 15.2]{Kech_setth}, we get that $s(E_l)\subseteq \gpd^0$ is Borel; and hence so is $s(E_l)^c:=\gpd^0\setminus s(E_l)$. Note that $\chi_E\chi_{s(E_l)^c}=0$. So $\psi_l:=\chi_{E_l}+\chi_{s(E_l)^c}$ is an element of $\mathbb{C}[X_{G,H}\rtimes N]$ and satisfies $l\cdot(\chi_E\otimes f)=l\chi_E\otimes f=\chi_E\chi_{E_l}\otimes f=\chi_E\psi_l\otimes f=\chi_E\otimes \psi_l f$.
                
	    However, for $(x,1)\in \gpd^0$ it holds that
        \begin{align*}
            \epsilon(\psi_{l} f)(x,1)&=\sum_{k\in N}\psi_{l} f(x,k)\\
            &=\sum_{k\in N}\sum_{(x_2,k_2)(x_1,k_1)=(x,k)}\psi_{l}(x_1,k_1)f(x_2,k_2)\\
            &=\sum_{k\in N}\psi_{l}(x,k)f(x,1)\\
            &=f(x,1)
        \end{align*}
	    where the last equality follows by the construction of $\psi_l$. Therefore, putting all together, $l$ acts trivially on $\chi_E\otimes f$ as we wanted.
    \end{proof}

    \begin{prop}\label{prop: vanish_dim_mod_criteria_for_gpds}
        With the above notation, let $M$ be a left $\C[\gpd_n]$-module in which $N$ acts trivially and $k$ a non-negative integer. Suppose that $\beta_i^{(2)}(N)$ vanishes for all $0\leq i\leq k$. Then
        \[
        \dim_{\scalebox{0.8}{$\VN(\gpd_n)$}}\left(\Tor_i^{\scalebox{0.8}{$\C[\gpd_n]$}}(\VN(\gpd_n),M)\right)=0
        \]
        for all $0\leq i \leq k$.
    \end{prop}

    \begin{proof}
            First note that since $N$ acts trivially on $M$, we have the following short exact sequence of left $\C[\gpd_n]$-modules
            \[
            0\rightarrow \ker \pi\rightarrow (\C[\gpd_n]\otimes_{\scalebox{0.8}{$\essbdd\ast N$}}L^{\infty}(\gpd^0))^{I}\xrightarrow{\pi} M\rightarrow 0
            \]
            for some cardinal number $I$. We argue by induction on $k$. For $k=0$, extending scalars on the above sequence yields to a surjective map
            \[
            \VN(\gpd_n)\otimes_{\scalebox{0.8}{$\C[\gpd_n]$}}(\C[\gpd_n]\otimes_{\scalebox{0.8}{$\essbdd\ast N$}}L^{\infty}(\gpd^0))^{I}\rightarrow \VN(\gpd_n)\otimes_{\scalebox{0.8}{$\C[\gpd_n]$}}M.
            \]
            Thus by \cref{defn: BSMRF} 
            \[
            \dim_{\scalebox{0.8}{$\VN(\gpd_n)$}}\left(\VN(\gpd_n)\otimes_{\scalebox{0.8}{$\C[\gpd_n]$}}M\right)
            \]
            is at most
            \[
            \dim_{\scalebox{0.8}{$\VN(\gpd_n)$}}\left(\VN(\gpd_n)\otimes_{\scalebox{0.8}{$\C[\gpd_n]$}}(\C[\gpd_n]\otimes_{\scalebox{0.8}{$\essbdd\ast N$}}L^{\infty}(\gpd^0))^{I}\right).
            \]
            \begin{claim}\label{claim: cgn_homomph_tensor_N}
                The natural left $\C[\gpd_n]$-homomorphism $\rho:\C[\gpd_n]\otimes_{\scalebox{0.8}{$\essbdd\ast N$}}L^{\infty}(\gpd^0)$ $\rightarrow \C[\gpd_n]\otimes_{\scalebox{0.8}{$\mathbb{C}[X_{G,H}\rtimes N]$}}L^{\infty}(\gpd^0)$ is a left $\dim_{\scalebox{0.8}{$L^{\infty}(\gpd^0)$}}$-isomorphism.
            \end{claim}
            \begin{proof}
                Apply \cref{thm: dim_isomph_changing_scalars} to $\C[\gpd_n]\otimes_{\scalebox{0.8}{$\essbdd\ast N$}}L^{\infty}(\gpd^0)$ and $\C[\gpd_n]\otimes_{\scalebox{0.8}{$\mathbb{C}[X_{G,H}\rtimes N]$}}L^{\infty}(\gpd^0)$. 
            \end{proof}
            \begin{claim}\label{claim: equality_betti_subgpd_n}
                For every non-negative integer $i$ it holds that 
                \[
                \dim_{\scalebox{0.8}{$\VN(\gpd_n)$}}\left(\Tor_i^{\scalebox{0.8}{$\C[\gpd_n]$}}(\VN(\gpd_n),\C[\gpd_n]\otimes_{\scalebox{0.8}{$\essbdd\ast N$}}L^{\infty}(\gpd^0))\right)=\beta_i^{(2)}(N).
                \]
            \end{claim}
            \begin{proof}
                By \cref{claim: cgn_homomph_tensor_N}, we make use of \cref{lem: dim_isomph_tensor_left} to get that
                \[
                \dim_{\scalebox{0.8}{$\VN(\gpd_n)$}}\left(\Tor_i^{\scalebox{0.8}{$\C[\gpd_n]$}}(\VN(\gpd_n),\C[\gpd_n]\otimes_{\scalebox{0.8}{$\essbdd\ast N$}}L^{\infty}(\gpd^0))\right)
                \]
                equals to
                \[
                \dim_{\scalebox{0.8}{$\VN(\gpd_n)$}}\left(\Tor_i^{\scalebox{0.8}{$\C[\gpd_n]$}}(\VN(\gpd_n),\C[\gpd_n]\otimes_{\scalebox{0.8}{$\mathbb{C}[X_{G,H}\rtimes N]$}}L^{\infty}(\gpd^0))\right).
                \]
                To conclude, note that the proof of \cref{prop: shapiro_dim_for_gpds} still holds in this case. Hence
                \[
                \dim_{\scalebox{0.8}{$\VN(\gpd_n)$}}\left(\Tor_i^{\scalebox{0.8}{$\C[\gpd_n]$}}(\VN(\gpd_n),\C[\gpd_n]\otimes_{\scalebox{0.8}{$\mathbb{C}[X_{G,H}\rtimes N]$}}L^{\infty}(\gpd^0))\right)=\beta_i^{(2)}(N).
                \]
                \end{proof}
            Thus, by exactness of dimension, \cref{claim: equality_betti_subgpd_n} and the fact that $\beta_0^{(2)}(N)$ vanishes, we conclude that 
            \[
            \dim_{\scalebox{0.8}{$\VN(\gpd_n)$}}(\VN(\gpd_n)\otimes_{\scalebox{0.8}{$\C[\gpd_n]$}}M)=0.
            \]
            Now assume the statement holds for $k-1$. Consider the long exact sequence associated to the initial exact sequence
            \begin{align*}
                &\Tor_k^{\scalebox{0.8}{$\C[\gpd_n]$}}(\VN(\gpd_n),(\C[\gpd_n]\otimes_{\scalebox{0.8}{$\essbdd\ast N$}}L^{\infty}(\gpd^0))^{I})\rightarrow\Tor_k^{\scalebox{0.8}{$\C[\gpd_n]$}}(\VN(\gpd_n),M)\rightarrow\ldots \\
                &\ldots\rightarrow \VN(\gpd_n)\otimes_{\scalebox{0.8}{$\C[\gpd_n]$}}\ker \pi\rightarrow \VN(\gpd_n)\otimes_{\scalebox{0.8}{$\C[\gpd_n]$}}(\C[\gpd_n]\otimes_{\scalebox{0.8}{$\essbdd\ast N$}}L^{\infty}(\gpd^0))^{I}\\
                &\rightarrow \VN(\gpd_n)\otimes_{\scalebox{0.8}{$\C[\gpd_n]$}}M\rightarrow 0.
		  \end{align*}
   
            In order to apply our induction hypothesis, we need that $N$ acts trivially on $\ker \pi$. This may not be the case though, but we can slightly modify our situation to a suitable one. Note that by \cref{claim: cgn_homomph_tensor_N} and exactness of dimension, the $\C[\gpd_n]$-homomorphism $\rho':(\C[\gpd_n]\otimes_{\scalebox{0.8}{$\essbdd\ast N$}}L^{\infty}(\gpd^0))^{I}\rightarrow (\C[\gpd_n]\otimes_{\scalebox{0.8}{$\C[X_{G,H}\rtimes N]$}}L^{\infty}(\gpd^0))^{I}$ is still a left $\dim_{\scalebox{0.8}{$\essbdd$}}$-isomorphism. Hence, $\pi(\ker \rho')$ is a $\C[\gpd_n]$-submodule of $M$ whose $\dim_{\scalebox{0.8}{$L^{\infty}(\gpd^0)$}}$ is zero. In other words, $M\rightarrow M/\pi(\ker \rho')$ is a left $\C[\gpd_n]$-homomorphism which is also a left $\dim_{\scalebox{0.8}{$L^{\infty}(\gpd^0)$}}$-isomorphism. Then we have the following commutative diagram of left $\C[\gpd_n]$-modules
            \[
            \begin{tikzcd}
                0\arrow{r}&\ker \pi\arrow{r}\arrow{d}&(\C[\gpd_n]\otimes_{\scalebox{0.8}{$\essbdd\ast N$}}L^{\infty}(\gpd^0))^{I}\arrow{r}{\pi}\arrow{d}{\rho'}&M\arrow{r}\arrow{d}&0\\
                0\arrow{r}&\ker \pi'\arrow{r}&(\C[\gpd_n]\otimes_{\scalebox{0.8}{$\mathbb{C}[X_{G,H}\rtimes N]$}}L^{\infty}(\gpd^0))^{I}\arrow{r}{\pi'}&M/\pi(\ker \rho')\arrow{r}&0
            \end{tikzcd}
            \]
            where $\pi'$ is induced from $\pi$. Since the central and right maps are left $\dim_{\scalebox{0.8}{$L^{\infty}(\gpd^0)$}}$-isomorphisms, so is $\ker \pi \rightarrow \ker \pi'$.	Thus, another use of \cref{lem: dim_isomph_tensor_left} gives us that
            \begin{align}\label{eqtn: dim_equality_ker}
                \dim_{\scalebox{0.8}{$\VN(\gpd_n)$}}\left(\Tor_i^{\scalebox{0.8}{$\C[\gpd_n]$}}(\VN(\gpd_n),\ker \pi)\right)=\dim_{\scalebox{0.8}{$\VN(\gpd_n)$}}\left(\Tor_i^{\scalebox{0.8}{$\C[\gpd_n]$}}(\VN(\gpd_n),\ker \pi')\right)
            \end{align}
            for all $0\leq i\leq k-1$. Moreover, according to \cref{lem: trivial_action} $N$ acts trivially on $(\C[\gpd_n]\otimes_{\scalebox{0.8}{$\mathbb{C}[X_{G,H}\rtimes N]$}}L^{\infty}(\gpd^0))^{I}$. Thus, $\ker \pi'$ is a left $\C[\gpd_n]$-module in which $N$ acts trivially; and hence by induction
            \[
            \dim_{\scalebox{0.8}{$\VN(\gpd_n)$}}\left(\Tor_i^{\scalebox{0.8}{$\C[\gpd_n]$}}(\VN(\gpd_n),\ker \pi')\right)=0
            \]
            for all $0\leq i\leq k-1$. Therefore, by \cref{claim: equality_betti_subgpd_n}, (\ref{eqtn: dim_equality_ker}) and the induction hypothesis, we conclude that all the dimensions below the $k$th row in the long exact sequence vanish. So, since the dimension is exact
            \[
            \dim_{\scalebox{0.8}{$\VN(\gpd_n)$}}\left(\Tor_k^{\scalebox{0.8}{$\C[\gpd_n]$}}(\VN(\gpd_n),M)\right)
            \]
            is at most
            \[
            \dim_{\scalebox{0.8}{$\VN(\gpd_n)$}}\left(\Tor_k^{\scalebox{0.8}{$\C[\gpd_n]$}}(\VN(\gpd_n),(\C[\gpd_n]\otimes_{\scalebox{0.8}{$\essbdd\ast N$}}L^{\infty}(\gpd^0))^{I})\right).
            \]
            Finally, exactness of dimension, \cref{claim: equality_betti_subgpd_n} and the fact that $\beta_k^{(2)}(N)$ vanishes, finish the proof.
    \end{proof}    
            	    
    We are now ready to show (\ref{eqtn: mult_betti_Hgpd_and_gpd_n}) and hence prove \cref{thm: groupoids_beta_k}.
  
    \begin{thm}\label{thm: control_dim_gpds}
        Let $G$ be a countable group, $H$ an infinite index subgroup in $G$ containing a normal subgroup $N$ in $G$ and $k$ a positive integer. Suppose that $\beta_i^{(2)}(N)$ vanishes for all $0\leq i\leq  k-1$. Then $\beta_k^{(2)}(H)\geq \beta_k^{(2)}(\gpd_n)$.
    \end{thm}
		
    \begin{proof}
	Consider the following short exact sequence of $\C[\gpd_n]$-modules
        \[
        0\rightarrow \ker \pi\rightarrow \C[\gpd_n]\otimes_{\scalebox{0.8}{$\C[\Hgpd]$}}L^{\infty} (\gpd^0)\xrightarrow{\pi} L^{\infty}(\gpd^0)\rightarrow 0.
        \]
        Extending scalars yields to the following long exact sequence	
        \begin{align*}
            &\Tor_k^{\scalebox{0.8}{$\C[\gpd_n]$}}\left(\VN(\gpd_n),\C[\gpd_n]\otimes_{\scalebox{0.8}{$\C[\Hgpd]$}}L^{\infty}(\gpd^0)\right)\rightarrow\Tor_k^{\scalebox{0.8}{$\C[\gpd_n]$}}\left(\VN(\gpd_n),L^{\infty}(\gpd^0)\right)\rightarrow \ldots\\
            &\ldots \rightarrow \VN(\gpd_n)\otimes_{\scalebox{0.8}{$\C[\gpd_n]$}}\ker \pi\rightarrow \VN(\gpd_n)\otimes_{\scalebox{0.8}{$\C[\gpd_n]$}}\C[\gpd_n]\otimes_{\scalebox{0.8}{$\C[\Hgpd]$}}L^{\infty}(\gpd^0)\\
            &\rightarrow \VN(\gpd_n)\otimes_{\scalebox{0.8}{$\C[\gpd_n]$}}L^{\infty}(\gpd^0)\rightarrow 0.
	\end{align*}
        From \cref{lem: trivial_action} follows that $N$ acts trivially on the left $\C[\gpd_n]$-module $\C[\gpd_n]\otimes_{\scalebox{0.8}{$\C[\Hgpd]$}}\essbdd$. Hence all dimensions below the $k$th row vanish according to \cref{prop: vanish_dim_mod_criteria_for_gpds}; and so by exactness of dimension we get
        \[
        \beta_k^{(2)}(H)\geq \beta_k^{(2)}(\gpd_n)
        \]
	where the left hand side comes from \cref{prop: shapiro_dim_for_gpds}.
    \end{proof}

    As a consequence we can also strengthen the known control results so far (see \cite[Theorem 7.2.(2)]{Luck02} and \cite[Theorem 5.6]{PetersonThom_Freiheit}).
    
    \begin{cor}\label{cor: Luck_beta_k_for_H}
         Let $G$ be a countable group, $H$ a subgroup in $G$ containing a normal subgroup $N$ in $G$ and $k$ a positive integer. Suppose that $\beta_i^{(2)}(N)$ vanishes for all $0\leq i\leq  k-1$. Then $\beta_k^{(2)}(H)\geq \beta_k^{(2)}(G)$.
    \end{cor}

    \begin{proof}
        Note that taking $n=1$ in \cref{thm: control_dim_gpds} yields to $\beta_k^{(2)}(H)\geq \beta_k^{(2)}(G)$ by \cref{lem: equality_betti_numb_gps_gpds}.
    \end{proof}

    \begin{rem}\label{rem: Luck_extended_criterion}
        Observe that it is necessary that the $L^2$-Betti numbers of $N$ vanish, take for instance the direct product of $k$ free groups on two generators.
    \end{rem}

    Moreover, \cref{cor: Luck_beta_k_for_H} can be proved, actually improved, directly in the setting of groups as we show now.

    \begin{proof}[Proof of \cref{thm: control_beta_k_subnormal_for_H}]
        Consider the following short exact sequence
        \[
        0\rightarrow I_G/I_H^G \rightarrow \C[G]\otimes_{\scalebox{0.8}{$\mathbb{C}[H]$}}\mathbb{C}\rightarrow \mathbb{C}\rightarrow 0
        \]
        which induces the long exact sequence
        \begin{align*}
            &\Tor_k^{\scalebox{0.8}{$\C[G]$}}(\UO(G),\C[G]\otimes_{\scalebox{0.8}{$\mathbb{C}[H]$}}\mathbb{C})\rightarrow \Tor_k^{\scalebox{0.8}{$\C[G]$}}(\UO(G),\mathbb{C})\rightarrow \ldots\\
            \ldots&\rightarrow \UO(G)\otimes_{\scalebox{0.8}{$\C[G]$}}I_G/I_H^G\rightarrow \UO(G)\otimes_{\scalebox{0.8}{$\mathbb{C}[G]$}}(\C[G]\otimes_{\scalebox{0.8}{$\mathbb{C}[H]$}}\mathbb{C})\rightarrow \UO(G)\otimes_{\scalebox{0.8}{$\mathbb{C}[G]$}}\mathbb{C}\rightarrow 0.
        \end{align*}
        We shall show that all the dimensions below the $k$th row vanish.
        \begin{claim}\label{claim: control_augment_ideal_1}
            If $N$ is a subnormal subgroup in $G$ with $\beta_i^{(2)}(N)=0$ for all $0\leq i\leq k-1$, then $\dim_{\scalebox{0.8}{$\UO(G)$}}(\Tor_i^{\scalebox{0.8}{$\C[G]$}}(\UO(G),I_G/I_N^G))=0$ for all $0\leq i\leq k-1$.
        \end{claim}
        \begin{proof}
            We work by induction on the depth of $N$. If $N$ is normal in $G$, then this follows from \cref{prop: control_dim_modules} since $N$ acts trivially on $I_G/I_N^G$. Thus suppose that $N\triangleleft N_1$ with $N_1$ a proper subnormal subgroup in $G$ of strictly less depth than $N$. Now consider the next exact sequence
            \begin{equation}\label{eqtn: control_augment_ideal_1}
                0\rightarrow I_{N_1}^G/I_N^G\rightarrow I_G/I_N^G\rightarrow I_G/I_{N_1}^G.
            \end{equation}
            Note that $\beta_i^{(2)}(N_1)=0$ for all $0\leq i\leq k-1$ by \cite[Theorem 7.2.(2)]{Luck02}; and hence by induction hypothesis $\dim_{\scalebox{0.8}{$\UO(G)$}}(\Tor_i^{\scalebox{0.8}{$\C[G]$}}(\UO(G),I_G/I_{N_1}^G))=0$ for all $0\leq i\leq k-1$. In addition, according to \cref{lem: isomph_ind_augment_ideal}, $\C[G]\otimes_{\scalebox{0.8}{$\C[N_1]$}}I_{N_1}/I_N^{N_1}$ is isomorphic to $I_{N_1}^G/I_N^G$. Hence, applying Shapiro's lemma along with (\ref{eqtn: preserve_dim_gps}) we get
            \[
            \dim_{\scalebox{0.8}{$\UO(G)$}}\left(\Tor_i^{\scalebox{0.8}{$\C[G]$}}(\UO(G),I_{N_1}^G/I_N^G)\right)=\dim_{\scalebox{0.8}{$\UO(N_1)$}}\left(\Tor_i^{\scalebox{0.8}{$\C[N_1]$}}(\UO(N_1),I_{N_1}/I_N^{N_1})\right)
            \]
            which is zero for all $0\leq i \leq k-1$ by the base step. Therefore, extending scalars on (\ref{eqtn: control_augment_ideal_1}) yields to the vanishing result and shows the induction.
        \end{proof}

        \begin{claim}\label{claim: control_augment_ideal_2}
            We have that $\dim_{\scalebox{0.8}{$\UO(G)$}}(\Tor_i^{\scalebox{0.8}{$\C[G]$}}(\UO(G),I_G/I_H^G))=0$ for all $0\leq i\leq k-1$.
        \end{claim}
        \begin{proof}
            Consider the next short exact sequence
            \begin{equation}\label{eqtn: control_augment_ideal_2}
                0\rightarrow I_H^G/I_N^G \rightarrow I_G/I_N^G \rightarrow I_G/I_H^G\rightarrow 0.
            \end{equation}
            Again by \cref{lem: isomph_ind_augment_ideal} $\C[G]\otimes_{\scalebox{0.8}{$\C[H]$}}I_{H}/I_N^{H}$ is isomorphic to $I_{H}^G/I_N^G$. Hence, applying Shapiro's lemma along with (\ref{eqtn: preserve_dim_gps}) we obtain
            \[
            \dim_{\scalebox{0.8}{$\UO(G)$}}\left(\Tor_i^{\scalebox{0.8}{$\C[G]$}}(\UO(G),I_{H}^G/I_N^G)\right)=\dim_{\scalebox{0.8}{$\UO(H)$}}\left(\Tor_i^{\scalebox{0.8}{$\C[H]$}}(\UO(H),I_{H}/I_N^{H})\right).
            \]
            Thus, since $N$ is also subnormal in $H$, both $\dim_{\scalebox{0.8}{$\UO(G)$}}\left(\Tor_i^{\scalebox{0.8}{$\C[G]$}}(\UO(G),I_{H}^G/I_N^G)\right)$ and $\dim_{\scalebox{0.8}{$\UO(G)$}}\left(\Tor_i^{\scalebox{0.8}{$\C[G]$}}(\UO(G),I_G/I_N^G)\right)$ vanish for all $0\leq i\leq k-1$ according to \cref{claim: control_augment_ideal_1}. Therefore, extending scalars on (\ref{eqtn: control_augment_ideal_2}) yields to the vanishing result.
        \end{proof}
        Now, by repeated applications of \cite[Theorem 7.2.(2)]{Luck02} it holds that $\beta_i^{(2)}(H)$ $=\beta_i^{(2)}(G)=0$ for all $0\leq i\leq k-1$. This together with \cref{claim: control_augment_ideal_2} shows that all dimensions below the $k$th row vanish as we wanted. In particular, by exactness of dimension we get
        \begin{align*}
            \beta_k^{(2)}(H)=&\dim_{\scalebox{0.8}{$\UO(G)$}}\left(\Tor_k^{\scalebox{0.8}{$\C[G]$}}(\UO(G),\C[G]\otimes_{\scalebox{0.8}{$\mathbb{C}[H]$}}\mathbb{C})\right)\\
            \geq&\dim_{\scalebox{0.8}{$\UO(G)$}}\left(\Tor_k^{\scalebox{0.8}{$\C[G]$}}(\UO(G),\mathbb{C})\right)=\beta_k^{(2)}(G).
        \end{align*}
    \end{proof}

    Note that the case $k=1$ of \cref{thm: control_beta_k_subnormal_for_H} provides evidence of a positive answer for Hillman's question. Indeed, we would be done if there were finite index subgroups in $G$ which had arbitrary high index and contained $N$. Finally we exhibit the progress made with respect to the depth of the subnormal subgroup $N$.

    \begin{proof}[Proof of \cref{cor: Hillman_depth_2}]
        Let $N=N_0\trianglelefteq N_1 \triangleleft G$ be a minimal subnormal chain for $N$. The case $N=N_1$ is \cref{thm: groupoids_beta_k}, so assume that $N\neq N_1$. First observe that $\beta_i^{(2)}(N_1)=0$ for all $0\leq i\leq k-1$ by \cite[Theorem 7.2.(2)]{Luck02}. Now set $H_1:=HN_1$. Note that if $|H_1:H|$ is finite, then $\beta_k^{(2)}(H_1)$ is finite. So according to \cref{thm: groupoids_beta_k} for $H_1$ and $N_1$, it follows that $\beta_k^{(2)}(G)$ vanishes. Thus suppose now that $|H_1:H|$ is infinite. Observe that
        \[
        \mbox{core}_{H_1}(H)=\mbox{core}_{N_1}(H)\supseteq\mbox{core}_{N_1}(N)=N.
        \]
        Thus, since $N$ is also subnormal on $\mbox{core}_{H_1}(H)$, repeated applications of \cite[Theorem 7.2.(2)]{Luck02} yield to $\beta_i^{(2)}(\mbox{core}_{H_1}(H))=0$ for all $0\leq i\leq k-1$. So $H$ contains a normal subgroup of $H_1$ with vanishing $L^2$-Betti numbers. Hence, according to \cref{thm: groupoids_beta_k} $\beta_k^{(2)}(H_1)=0$. Therefore, by \cref{cor: Luck_beta_k_for_H}, $0=\beta_k^{(2)}(H_1)\geq \beta_k^{(2)}(G)$ which ends the proof.
    \end{proof}

    The last corollary is a generalization of the classical results by A. Karras and D. Solitar in \cite{KarrasSolitar}, H. Griffiths in \cite{Griffiths}, and B. Baumslag in \cite{Baumslag_FreeProd}.

    \begin{cor}
        Let $G$ be a countable group with positive $\beta_1^{(2)}(G)$, $H$ a finitely generated subgroup in $G$ and $N$ a normal subgroup in $G$. Suppose that $H$ contains an infinite normal subgroup in $N$, then $H$ is of finite index in $G$.
    \end{cor}

\bibliographystyle{alpha}
\bibliography{bib}

\end{document}